\documentclass[a4paper,10pt]{article}
\usepackage[utf8x]{inputenc}
\usepackage{tracefnt,amsmath,tabu,array}
\usepackage{amssymb,graphicx,setspace,amsfonts,amsbsy}
\usepackage{pifont,latexsym,ifthen,amsthm,rotating,calc,textcase,booktabs,cancel,slashed}

\newtheorem{theorem}{Theorem}[section]
\newtheorem{lemma}[theorem]{Lemma}
\newtheorem{corollary}[theorem]{Corollary}
\newtheorem{proposition}[theorem]{Proposition}

\newtheorem{remark}[theorem]{Remark}
\newcommand{\filledbox}{\leavevmode
  \hbox to.77778em{%
  \hfil\vbox to.675em{\hrule width.6em height.6em}\hfil}}

\newcommand{\Rm}{{\mathbb R}}

\newcommand{\eps}{\varepsilon}

\begin{document}
\tabulinesep=1.0mm
\title{Energy Distribution of solutions to defocusing semi-linear wave equation in two dimensional space \footnote{MSC classes: 35L05, 35L71.}}

\author{Liang Li, Ruipeng Shen and Lijuan Wei\\
Centre for Applied Mathematics\\
Tianjin University\\
Tianjin, China
}

\maketitle

\begin{abstract}
  We consider finite-energy solutions to the defocusing nonlinear wave equation in two dimensional space. We prove that almost all energy moves to the infinity at almost the light speed as time tends to infinity. In addition, the inward/outward part of energy gradually vanishes as time tends to positive/negative infinity. These behaviours resemble those of free waves. We also prove some decay estimates of the solutions if the initial data decay at a certain rate as the spatial variable tends to infinity. As an application, we prove a couple of scattering results for solutions whose initial data are in a weighted energy space. Our assumption on decay rate of initial data is weaker than previous known scattering results. 
 \end{abstract}

\section{Introduction}
In this work we consider the defocusing nonlinear wave equation in 2-dimensional space
\[
 \left\{\begin{array}{ll} \partial_t^2 u - \Delta u = - |u|^{p-1} u, & (x,t)\in \Rm^2 \times \Rm; \\
 u|_{t=0} = u_0; &\\
 \partial_t|_{t=0} = u_1 . &  \end{array}\right. (CP1)
\]
The conserved energy is defined by
\[
 E = \int_{\Rm^2} \left(\frac{1}{2}|\nabla u(x,t)|^2 + \frac{1}{2}|u_t(x,t)|^2 + \frac{1}{p+1}|u(x,t)|^{p+1} \right) dx.
\]
\subsection{Background}
Defocusing nonlinear wave equations 
\[
 \partial_t^2 u - \Delta u = - |u|^{p-1} u, \qquad (x,t)\in\Rm^d \times \Rm
\]
have been extensively studied in the past few decades. This equation is invariant under a natural rescaling. Namely, if $u$ is a solution with initial data $(u_0, u_1)$, then 
\[ 
 u_\lambda (x,t) = \lambda^{-2/(p-1)} u(x/\lambda, t/\lambda)
\]
is another solution to the same equation with initial data 
\[
 (u_{0,\lambda}, u_{1,\lambda})=\left(\lambda^{-2/(p-1)} u_0 (\cdot/\lambda), \lambda^{-2/(p-1)-1} u_1(\cdot/\lambda)\right).
\] 
A basic calculation shows that $(u_{0,\lambda}, u_{1,\lambda})$ share the same $\dot{H}^{s_p}\times \dot{H}^{s_p-1}(\Rm^d)$ norm as the original initial data if we choose $s_p = d/2-2/(p-1)$. Thus the space $\dot{H}^{s_p}\times \dot{H}^{s_p-1}(\Rm^d)$ is usually called the critical Sobolev space of this equation. 

\paragraph{Local theory} The existence and uniqueness of solutions to semi-linear wave equation like (CP1) follows a combination of suitable Strichartz estimates and a fixed-point argument. Readers may refer to, for example, Kapitanski \cite{loc1} and Lindblad-Sogge \cite{ls} for more details. In this work we mainly consider finite-energy solutions in the 2-dimensional case. The global well-posedness of these solutions has been proved in Gibibre-Velo \cite{globalwell}. Therefore we focus on global, especially asymptotic behaviours of solutions to (CP1). We first give a brief review on previously known results concerning the global behaviour of solutions to defocusing semi-linear wave equations, both in higher dimensions $d\geq 3$ and dimension 2.

\paragraph{Previous results in higher dimensions} The global behaviour of solutions with energy critical nonlinearity $p_e=1+4/(d-2)$ in 3 or higher dimensional space  has been well understood. It was proved in the last few decades of 20th century that any solution with a finite energy must exist globally in time and scatter in both time directions. Please see, for instance, Bahouri-G\'{e}rard \cite{bahouri}, Bahouri-Shatah \cite{ascattering}, Ginibre-Soffer-Velo \cite{locad1}, Grillakis \cite{mg1,mg2}, Kapitanski \cite{continuousL}, Nakanishi \cite{enscatter1, enscatter2}, Pecher \cite{local1}, Shatah-Struwe \cite{ss2} and Struwe \cite{struwe}. The energy sub-critical case $p<p_e$ and energy super-critical case $p>p_e$ have also been discussed. For example, there are many conditional scattering results, in different dimensions and for different ranges of $p$, proving that if the critical Sobolev norm of a solution is uniformly bounded in the whole lifespan, then this solution must be a global solution and scatter.  Please see, for instance, Bulut \cite{ab1}, Dodson-Lawrie \cite{cubic3dwave}, Dodson et al. \cite{nonradial3p5}, Duyckaerts et al. \cite{dkm2}, Kenig-Merle \cite{km}, Killip-Visan \cite{kv2, kv3}, Rodriguez \cite{sub45} and Shen \cite{shen2}. There are also scattering results only depending on the information about initial data. For example, Ginibre-Velo \cite{conformal2} proved the scattering of solutions by conformal conservation law in the energy sub-critical range if the initial is contained in a weighted Sobolev space. Recently the second author \cite{shenenergy, shen3dnonradial} introduced an inward/outward energy theory and proved the scattering of solutions for initial data in a larger weighted Sobolev space, by considering the energy distribution properties of solutions and/or weighted Morawetz estimates.

\paragraph{Previous results in dimension 2} Since the energy critical exponent is $p=+\infty$ in dimension 2, the equation (CP1) is always energy sub-critical for all $p \in (1,+\infty)$. In general, there are less results available in dimension 2. Theoretically speaking, it is more difficult to obtain scattering result in dimension 2 than higher dimensions. It is because the dispersive rate of linear wave equation depends on the dimension. It is well known that if the initial data are smooth and compactly supported, then the corresponding solution to linear homogenous wave equation satisfies $|u(x,t)| \lesssim t^{(d-1)/2}$. The decay rate is lower in dimension 2 than higher dimensions. Many previously known scattering results in dimension 2 are based on the conformal conversation laws. Please see Gibibre-Velo \cite{conformal2}, Glassey-Pecher \cite{subconformal2d}, Hidano \cite{conformal} and Wei-Yang \cite{yang2d}, for example. These works assume that the initial data satisfy 
\[
  \int_{\Rm^2} \left[(|x|^2+1) (|\nabla u_0 (x)|^2 + |u_1(x)|^2) + |u_0(x)|^{p+1} \right] dx < \infty.
 \]
In a joint work \cite{subhyper} with Staffilani, the second author applies a transformation (introduced in Tataru \cite{tataru}) between wave equation in Euclidean space and shifted wave equation in hyperbolic space and proves the scattering of solutions to the quintic equation $p=5$, if the initial data are radial and satisfy 
\begin{align*}
 &|\nabla u_0 (x)|, |u_1(x)| \lesssim (1+|x|)^{-3/2-\eps},& &|u_0(x)|\lesssim (1+|x|)^{-1/2-\eps},& &\eps>0.&
\end{align*}
Tsutaya \cite{ktsutaya} proves the same result for small but possibly non-radial initial data with a similar decay rate. 

\subsection{Motivation and main idea}

Inward/outward energy theory discusses the energy distribution properties of solutions in details and proves the scattering of solutions for initial data with a lower decay rate than previously known results, in 3 or higher dimensional space. In this work we consider how to use a similar idea to investigate global behaviour of solutions in dimension 2. We start by introducing a few notations for convenience. These notation will be used throughout this work.

\paragraph{Notations} In this work we will use the notation $e(x,t)$ for the energy density
\[
 e(x,t) = \frac{1}{2}|\nabla u(x,t)|^2 + \frac{1}{2} |u_t(x,t)|^2 + \frac{1}{p+1}|u(x,t)|^{p+1}.
\]
We use $u_r$, $\slashed{\nabla} u$ for the derivative in the radial direction and the covariant derivative on the circle centred at the origin, respectively. 
\begin{align*}
 &u_r(x,t) = \frac{x}{|x|} \cdot \nabla u(x,t);& &\slashed{\nabla} u =  \nabla u - u_r \frac{x}{|x|};& &|\nabla u|^2 = |u_r|^2 + |\slashed{\nabla} u|^2&
\end{align*}
We also define the weighted energy 
\[
 E_\kappa(u_0,u_1) = \int_{\Rm^2} (|x|^\kappa+1)\left(\frac{1}{2}|\nabla u_0(x)|^2 + \frac{1}{2}|u_1(x)|^2 + \frac{1}{p+1}|u_0(x)|^{p+1}\right) dx.
\]
Please note that this is no longer a conserved quantity in general. 

\paragraph{Morawetz estimates} The main tool of this work is still a Morawetz-type estimate. This kind of estimates were first found by Morawetz \cite{morawetz}  for wave/Klein-Gordon equations. Lin-Strauss \cite{morawetzsch} then generalized Morawetz estimates to Schr\"{o}ndinger equations. Colliander-Keel-Staffilani-Takaoka-Tao \cite{interaction} introduced interaction Morawetz estimates for Schr\"{o}ndinger equations. Nowadays the Morawetz estimate has been one of the most important tools in the study of dispersive equations. In the case of 2D wave equation, there are also previously known Morawetz estimates 
\begin{equation}
 \int_{1}^\infty \int_{\Rm^2} \frac{(t^2+|x|^2)|\slashed{\nabla} u|^2 + |ru_t+tu_r|^2 + t^2|u|^{p+1}}{(t^2+|x|^2)^{3/2}} dx dt \lesssim C(E). \label{Nakanishi morawetz}
\end{equation}
Nakanishi \cite{enscatter1} first introduced this kind of Morawetz estimate for non-coercive energy critical wave equations in dimension $d\geq 3$ by using the multiplier $\frac{(-t,x)}{\sqrt{t^2+|x|^2}}$, then generalized it to Klein-Gordon equation in lower dimensions $d=1,2$ in a subsequent work \cite{morawetzNLKG}. A similar argument shows that this kind of Morawetz estimate holds for 2D wave equation as well. For the convenience of readers, we give a brief proof of \eqref{Nakanishi morawetz} in the appendix by the method of Nakanishi. 
 
\paragraph{Main idea} The global integral estimate \eqref{Nakanishi morawetz} does give some information on the asymptotic behaviour of solutions. For example, since $1/t$ is not integrable, the inequality \eqref{Nakanishi morawetz} implies that the $L^{p+1}$ norm of $u(\cdot,t)$ must tend to zero (in the average sense) as $t\rightarrow +\infty$. In order to gain more useful information about asymptotic behaviour, we try to use another multiplier 
\[
 \nabla \Psi = \left\{\begin{array}{ll} x, & \hbox{if} \; |x|\leq R; \\ Rx/|x|, & \hbox{if}\; |x|\geq R;\end{array}\right.
\]
i.e. the same multiplier as we used in the higher dimensional case $d\geq 3$. However, a similar argument to the higher dimensional case can neither give a global space-time integral estimate nor develop an inward/outward energy theory, due to the presence of additional terms in the right hand side:
\[
 \hbox{Positive terms} = 2E + \int_{t_1}^{t_2} \int_{|x|>R} \frac{|u|^2}{4|x|^3} dx dt + \sum_{i=1,2} \int_{|x|>R} \frac{|u(x,t_i)|^2}{8|x|^2} dx.
\]
If $t_1\leq -R$ and $t_2\geq R$, then one of the terms in the left hand side is almost $2E$. We may move this major term to the right hand and obtain
\[
 \hbox{Positive terms}  \lesssim \int_{\Rm^2} \min\{|x|/R,1\} e(x,0) dt + \int_{t_1}^{t_2} \int_{|x|>R} \frac{|u|^2}{|x|^3} dx dt + \sum_{i=1,2} \int_{|x|>R} \frac{|u(x,t_i))|^2}{|x|^2} dx.
\]
Here the first integral in the right hand side converges to zero as $R\rightarrow +\infty$. In addition, if the initial data satisfy $E_\kappa (u_0,u_1)<+\infty$ for a constant $\kappa \in (0,1)$, then this integral decays faster than $R^{-\kappa}$. The other two terms in the right hand side can be dealt with in two different ways:
\begin{itemize}
\item We observe that the left hand side contains the following terms:
\begin{align*}
 &\int_{t_1}^{t_2} \int_{|x|>R} \frac{|u|^{p+1}}{|x|} dx dt,& &\int_{|x|>R} |u|^{p+1} dx.&
\end{align*}
If $|u|\gg |x|^{-2/(p-1)}$, then we have $|u|^2/|x|^3 \ll |u|^{p+1}/|x|$ and $|u|^2/|x|^2 \ll |u|^{p+1}$. In this case the additional terms in the right hand side can be absorbed by the left hand side. On the other hand, if $|u|\lesssim |x|^{-2/(p-1)}$, then the additional terms themselves are already small if $|x|$ is large. In summary we can always remove these terms at a reasonable cost.
\item If the initial data decay sufficiently fast as the spatial variable tends to infinity, we may find a better estimate of these additional terms by finite speed of propagation and a weighted Hardy inequality. 
\end{itemize}
We will calculate as carefully as we can when we work on a Morawetz-type estimate. A few positive terms, which were simply neglected in most previous works, turn out to be very useful ones when we discuss the energy distribution of solutions.

\subsection{Main Results}
Now we give the main results of this work. The first result is about the energy distribution of solutions to (CP1).
\begin{theorem} \label{main 1}
Assume $p > 3$. Let $u$ be a solution to (CP1) with a finite energy. Then 
 \begin{itemize}
  \item[(a)] The following limits hold as time tends to infinity
   \[
     \lim_{t\rightarrow \pm \infty} \int_{|x|<|t|} \frac{|t|-|x|}{|t|} e(x,t) dx = 0.
   \]
  \item[(b)] The inward/outward part of energy vanishes as time tends positive/negative infinity. 
  \[
   \lim_{t\rightarrow \pm \infty} \int_{\Rm^2} \left(|u_r\pm u_t|^2 + |\slashed{\nabla}u|^2 + |u|^{p+1} \right) dx = 0.
  \]
  \item[(c)] Furthermore, if the initial data satisfy $E_\kappa(u_0,u_1)< +\infty$ for a constant 
  \[ 
   \kappa \in \left\{\begin{array}{ll}  (0, 1), & \hbox{if} \; p\geq 5;\\ 
    (0,\frac{p-3}{2}), & \hbox{if} \; 3<p<5; \end{array}\right. 
  \]
 then we have the following decay estimates
  \begin{align*}
  \lim_{t\rightarrow \pm \infty} \int_{|x|<|t|} \frac{|t|-|x|}{|t|^{1-\kappa}} e(x,t) dx & = 0; \\
   \lim_{t\rightarrow \pm \infty} |t|^\kappa \int_{\Rm^2} \left(|u_r\pm u_t|^2 + |\slashed{\nabla}u|^2 + |u|^{p+1} \right) dx & = 0.
  \end{align*}
 \end{itemize}
\end{theorem}

\noindent As an application of our theory on energy distribution, we also have the following scattering results. 

\begin{theorem} \label{main 3} 
 Assume $p>5$. Let $u$ be a solution to (CP1) whose initial data satisfy $E_\kappa (u_0,u_1) < +\infty$ for a constant $\kappa > \frac{3p+5}{4p}$. Then the solution $u$ scatters in both two time directions. More precisely, there exist $(u_0^\pm ,u_1^\pm)\in (\dot{H}^1 \cap \dot{H}^{s_p}(\Rm^2)) \times (L^2 \cap \dot{H}^{s_p-1}(\Rm^2))$, so that 
 \[
  \lim_{t\rightarrow \pm \infty} \left\|\begin{pmatrix} u(\cdot,t)\\ u_t(\cdot,t) \end{pmatrix} - \mathbf{S}_L(t) \begin{pmatrix} u_0^\pm\\ u_1^\pm\end{pmatrix}\right\|_{\dot{H}^s \times \dot{H}^{s-1}(\Rm^2)} = 0. 
 \]
 holds for all $s\in [s_p,1]$. Here $\mathbf{S}_L$ is the linear wave propagation operator. 
\end{theorem}

\begin{theorem} \label{main 2}
Assume $p > 1+2\sqrt{3} \approx 4.464$. Let $u$ be a radial solution to (CP1) with a finite energy. Then 
 \begin{itemize}
  \item[(a)] There exist two radial finite-energy free waves $\tilde{u}^+, \tilde{u}^-$, so that 
 \[
  \lim_{t\rightarrow \pm \infty} \int_{|x|>|t|-\eta} \left(|\nabla \tilde{u}^\pm (x,t)-\nabla u(x,t)|^2+|\tilde{u}_t^\pm (x,t)-u_t(x,t)|^2\right) dx = 0, \quad \forall \eta\in \Rm. 
 \]
 \item[(b)] If the initial data $(u_0,u_1)$ satisfy $E_{\frac{4}{p+1}}(u_0,u_1) < +\infty$, then the solution $u$ scatters in both two time directions. 
  \[
  \lim_{t\rightarrow \pm \infty} \|(\tilde{u}^\pm(\cdot,t)-u(\cdot,t), \tilde{u}_t^\pm (\cdot,t)-u_t(\cdot,t))\|_{\dot{H}^1\times L^2(\Rm^2)} = 0.
 \]
 \end{itemize}
\end{theorem}

\begin{remark}
 We may call Part (a) of Theorem \ref{main 2} exterior scattering. The free waves $\tilde{u}^+$, $\tilde{u}_-$ are the radiation part of solutions $u$.  
\end{remark}

\begin{remark}
 The method of conformal conservation laws assumes that $E_2(u_0,u_1)<+\infty$. The initial data in works \cite{subhyper, ktsutaya} satisfy $E_{1+\eps} (u_0,u_1)<+\infty$. The decay rate of initial data in the scattering results of this work is lower than these previously known results. The initial data in Theorem \ref{main 2} are not necessary contained in the critical Sobolev space. 
\end{remark}

\paragraph{The structure of this work} This paper is organized as follows. We first give a few preliminary results and technical lemmata in Section 2. Then in Section 3 we discuss our Morawetz-type estimates in details. Section 4 is devoted to the proof of energy distribution properties. Finally we discuss scattering theory of non-radial and radial solutions in Section 5 and 6. 

\section{Preliminary Results}

\begin{lemma} \label{pointwise estimate}
 If $u \in \dot{H}^1(\Rm^2)\cap L^{p+1}(\Rm^2)$ is a radial function, then we have the pointwise estimate
 \[
  |u(r)| \lesssim_1 r^{-\frac{2}{p+3}} \|u\|_{\dot{H}^1}^{\frac{2}{p+3}} \|u\|_{L^{p+1}}^{\frac{p+1}{p+3}}, \qquad r>0.
 \]
\end{lemma}
\begin{proof}
We start by
\begin{align*}
  |u(r_2) - u(r_1)| & = \left|\int_{r_1}^{r_2} u_r (r) dr\right| \leq \left(\int_{r_1}^{r_2} r^{-1} dr\right)^{1/2} \left(\int_{r_1}^{r_2} r |u_r|^2 dr\right)^{1/2} \leq \left(\frac{r_2-r_1}{r_1}\right)^{1/2}  \|u\|_{\dot{H}^1}. 
 \end{align*}
Thus we have 
\[
 |u(r') - u(r)| \leq |u(r)|/2\Rightarrow |u(r')|\geq |u(r)|/2, \qquad \forall r\leq r'\leq r+ \frac{r|u(r)|^2}{4\|u\|_{\dot{H}^1}^2}.
\]
Combining this estimate with $L^{p+1}$ norm, we obtain
\[
 \|u\|_{L^{p+1}}^{p+1} \geq \int_{r<|x|<r+\frac{r|u(r)|^2}{4\|u\|_{\dot{H}^1}^2}} |u(x)|^{p+1} dx \geq 2\pi r \cdot \frac{r|u(r)|^2}{4\|u\|_{\dot{H}^1}^2} \cdot \left(\frac{|u(r)|}{2}\right)^{p+1}
\]
This immediately finishes the proof. 
\end{proof}

\begin{proposition}[Strichartz estimates, see Proposition 3.1 of Ginibre-Velo \cite{strichartz}]\label{Strichartz estimates} 
 Let $2\leq q_1,q_2 \leq \infty$, $2\leq r_1,r_2 < \infty$ and $\rho_1,\rho_2,s\in \Rm$ be constants with
 \begin{align*}
  &\frac{2}{q_i} + \frac{1}{r_i} \leq \frac{1}{2},& &i=1,2;& \\
  &\frac{1}{q_1} + \frac{2}{r_1} = 1 + \rho_1 - s;& &\frac{1}{q_2} + \frac{2}{r_2} = \rho_2 +s.&
 \end{align*}
 Assume that $u$ is the solution to the linear wave equation
\[
 \left\{\begin{array}{ll} \partial_t u - \Delta u = F(x,t), & (x,t) \in \Rm^2 \times [0,T];\\
 u|_{t=0} = u_0 \in \dot{H}^s; & \\
 \partial_t u|_{t=0} = u_1 \in \dot{H}^{s-1}. &
 \end{array}\right.
\]
Then we have
\begin{align*}
 \sup_{t\in [0,T]} \left\|\left(u(\cdot,t), \partial_t u(\cdot,t)\right)\right\|_{\dot{H}^s \times \dot{H}^{s-1}} & +\|D_x^{\rho_1} u\|_{L^{q_1} L^{r_1}([0,T]\times \Rm^2)} \\
 & \leq C\left(\left\|(u_0,u_1)\right\|_{\dot{H}^s \times \dot{H}^{s-1}} + \left\|D_x^{-\rho_2} F\right\|_{L^{\bar{q}_2} L^{\bar{r}_2} ([0,T]\times \Rm^2)}\right).
\end{align*}
Here the coefficients $\bar{q}_2$ and $\bar{r}_2$ satisfy $1/q_2 + 1/\bar{q}_2 = 1$, $1/r_2 + 1/\bar{r}_2 = 1$. The fractional differential operator $D_x^\rho$ is defined by the Fourier multiplier $|\xi|^\rho$. The constant $C$ does not depend on $T$ or $u$. 
\end{proposition}

\begin{remark}
 We call $(q_1,r_1)$ an $s$-admissible pair if $q_1,r_1,s$ and $\rho_1=0$ satisfies the conditions given above.  
\end{remark}

\paragraph{Chain rule} We also need the following ``chain rule'' for fractional derivatives. Please refer to Christ-Weinstein \cite{fchain}, 
and Taylor \cite{fchain4} for more details. 
\begin{lemma} \label{chain rule}
 Assume that a function $F$ satisfies $F(0) = F'(0) = 0$ and 
 \begin{align*}
  &|F'(a+b)| \leq C(|F'(a)| + |F'(b)|),& &|F''(a+b)| \leq C(|F''(a)| + |F''(b)|),&
 \end{align*}
 for all $a,b\in \Rm$. Then we have
 \[
  \|D^\alpha F(u)\|_{L^p (\Rm^2)} \leq C \|D^\alpha u\|_{L^{p_1} (\Rm^2)} \|F'(u)\|_{L^{p_2} (\Rm^2)}
 \]
 for $0<\alpha<1$ and $1/p = 1/p_1+1/p_2$, $1<p, p_1,p_2<\infty$.
\end{lemma}

\begin{proposition} \label{scattering criterion}
Assume $p \geq 5$.  Let $u$ be a solution to (CP1) with $(u_0,u_1) \in \dot{H}^{s_p} \times \dot{H}^{s_p-1}(\Rm^2)$. If $u$ can be defined for all $t\geq 0$ and scatters in the critical Sobolev space, i.e. there exists $(u_0^+,u_1^+) \in \dot{H}^{s_p}\times \dot{H}^{s_p-1}$ so that
\[
 \lim_{t\rightarrow +\infty} \left\|\begin{pmatrix} u(\cdot, t)\\ u_t(\cdot, t)\end{pmatrix} - \mathbf{S}_L(t) \begin{pmatrix} u_0^+\\ u_1^+\end{pmatrix} \right\|_{\dot{H}^{s_p}\times \dot{H}^{s_p-1}(\Rm^2)} = 0,
\]
 then $u \in L^q L^r ([0,\infty) \times \Rm^2)$ for all $s_p$-admissible pairs $(q,r)$. 
\end{proposition} 
\begin{proof} 
 The proof is similar to the higher dimensional case. For convenience of readers, we stretch a proof here. First of all, given any time $t_0$ and time interval $I$ containing $t_0$, we introduce the $Y(I)$ norm 
 \[
  \|u\|_{Y(I)} = \|u\|_{L^\frac{3p(p-1)}{3p-5} L^\frac{6p(p-1)}{3p+5} (I \times \Rm^2)}.
 \]
 By Strichartz estimates, the solution $u$ to the linear equation $\partial_t^2 u - \Delta u = F$ with initial data $(u_0,u_1)$ at time $t_0$ satisfies 
 \begin{align*}
  \|u\|_{Y(I)} & \leq \|\mathbf{S}_L(t-t_0) (u_0,u_1)\|_{Y(I)} + C_1 \|F\|_{L^{\frac{3(p-1)}{3p-5}} L^\frac{6(p-1)}{3p+5} (I\times \Rm^2)};\\
  \|\mathbf{S}_L(t-t_0) (u_0,u_1)\|_{Y(I)} & \leq C_2 \|(u_0,u_1)\|_{\dot{H}^{s_p} \times \dot{H}^{s_p-1}(\Rm^2)}.
 \end{align*} 
 In addition, we have 
\begin{align*}
 \left\|-|u|^{p-1} u\right\|_{L^{\frac{3(p-1)}{3p-5}} L^\frac{6(p-1)}{3p+5} (I \times \Rm^2)} & \leq \|u\|_{Y(I)}^p; \\
\left\||v|^{p-1}v-|u|^{p-1} u\right\|_{L^{\frac{3(p-1)}{3p-5}} L^\frac{6(p-1)}{3p+5} (I\times \Rm^2)} & \leq C_3 \|u-v\|_{Y(I)}(\|u\|_{Y(I)}^{p-1} + \|v\|_{Y(I)}^{p-1}).
\end{align*}
Given $u \in Y(I)$, we define $\mathbf{T} u$ to be the solution $v$ to linear wave equation $\partial_t^2 v - \Delta v = -|u|^{p-1} u$ with initial data $(u_0,u_1)$ at time $t_0$. Then the inequalities above imply
\begin{align*}
 \|\mathbf{T} u\|_{Y(I)} & \leq  \|\mathbf{S}_L(t-t_0) (u_0,u_1)\|_{Y(I)} + C_1 \|u\|_{Y(I)}^p; \\
 \|\mathbf{T} u - \mathbf{T} v\|_{Y(I)} & \leq C_3\left(\|u\|_{Y(I)}^{p-1} + \|v\|_{Y(I)}^{p-1}\right)  \|u-v\|_{Y(I)}.
\end{align*} 
Thus there exists a constant $\delta = \delta(p) >0$, so that if initial data $(u_0,u_1)$ satisfy 
\[
 \|\mathbf{S}_L(t-t_0) (u_0,u_1)\|_{Y(I)} \leq \delta,
\]
then the operator $\mathbf{T}$ becomes a contraction map on the metric space $\{u: \|u\|_{Y(I)} < 2\delta\}$. As a result there exists a unique solution $u$ to (CP1) in the time interval $I$ with $\|u\|_{Y(I)} \leq 2 \delta$. In addition, we have
\begin{align*}
 &\|\mathbf{S}_L(t-T) (u(\cdot,T),u_t(\cdot,T))\|_{Y([T,+\infty))}\\
 &\quad  \leq  \|u^+\|_{Y([T,+\infty))} + \|\mathbf{S}_L(t-T)(u(\cdot,T)-u^+(\cdot,T), u_t(\cdot,T)-u_t^+(\cdot,T))\|_{Y([T,+\infty))}\\
 &\quad \leq \|u^+\|_{Y([T,+\infty))} + C_2 \|(u(\cdot,T)-u^+(\cdot,T), u_t(\cdot,T)-u_t^+(\cdot,T))\|_{\dot{H}^{s_p}\times \dot{H}^{s_p-1}}.
\end{align*}
Here $u^+ = \mathbf{S}_L(t) (u_0^+, u_1^+)$ is the free wave with initial data $(u_0^+,u_1^+)$. We utilize the scattering assumption and the Strichartz estimates $\|u_+\|_{Y(\Rm^+)} < +\infty$ to conclude
\[
 \lim_{T\rightarrow +\infty} \|\mathbf{S}_L(t-T) (u(\cdot,T),u_t(\cdot,T))\|_{Y([T,+\infty))} = 0.
\]
Thus there exists a large time $T>0$ so that 
\[
 \|\mathbf{S}_L(t-T) (u(\cdot,T),u_t(\cdot,T))\|_{Y([T,+\infty))} \leq \delta,
\]
In addition, since the time interval $[0,T]$ is compact, we can also split it into finite small intervals $[0,T] = I_1 \cup I_2 \cup \cdots \cup I_m$ with a time $t_k \in I_k$ in each interval, so that 
\[
 \|\mathbf{S}_L(t-t_k) (u(\cdot,t_k),u_t(\cdot,t_k))\|_{Y(I_k)} \leq \delta.
\]
According to the local theory given above, we obtain 
\begin{align*}
 &\|u\|_{Y([T,+\infty))} < 2\delta;& &\|u\|_{Y(I_k)} < 2\delta.&
\end{align*}
Therefore we have $\|u\|_{Y([0,+\infty))} < + \infty$. Finally we apply the Strichartz estimates and conclude
\begin{align*}
 \|u\|_{L^q L^r ([0,\infty) \times \Rm^2)} & \lesssim \|\mathbf{S}_L(t) (u_0,u_1)\|_{L^q L^r([0,+\infty)\times \Rm^2)} + \left\|-|u|^{p-1}u\right\|_{L^{\frac{3(p-1)}{3p-5}} L^\frac{6(p-1)}{3p+5} ([0,+\infty) \times \Rm^2)}\\
 & \lesssim \|(u_0,u_1)\|_{\dot{H}^{s_p}\times \dot{H}^{s_p-1}(\Rm^2)} + \|u\|_{Y([0,+\infty))}^p < +\infty
\end{align*}
for any given $s_p$-admissible pair $(q,r)$. 
\end{proof}

\begin{theorem}[Radiation field] \label{radiation}
Let $u$ be a solution to the free wave equation $\partial_t^2 u - \Delta u = 0$ with initial data $(u_0,u_1) \in \dot{H}^1 \times L^2(\Rm^2)$. Then 
\[
 \lim_{t\rightarrow +\infty} \int_{\Rm^2} |\slashed{\nabla} u(x,t)|^2 dx = 0
\]
 and there exists a function $G_+ \in L^2(\Rm \times \mathbb{S}^1)$ so that
\begin{align*}
 \lim_{t\rightarrow +\infty} \int_0^\infty \int_{\mathbb{S}^1} \left|r^{1/2} \partial_t u(r\Theta, t) - G_+(r-t, \Theta)\right|^2 d\Theta dr &= 0;\\
 \lim_{t\rightarrow +\infty} \int_0^\infty \int_{\mathbb{S}^1} \left|r^{1/2} \partial_r u(r\Theta, t) + G_+(r-t, \Theta)\right|^2 d\Theta dr & = 0.
\end{align*}
In addition, the map $(u_0,u_1) \rightarrow \sqrt{2} G_+$ is a bijective isometry form $\dot{H}^1 \times L^2(\Rm^2)$ to $L^2 (\Rm \times \mathbb{S}^1)$.
\end{theorem}
The $3$-dimensional version of this theorem has been known many years ago. For example, please see Friedlander \cite{radiation1, radiation2}. Duyckaerts-Kenig-Merle \cite{dkm3} gives a proof for all dimensions $d\geq 3$. The $2$-dimensional case can be proved in almost the same way. In fact, only one of the ingredients in Duychaerts-Kenig-Merle's proof is unique in $3$ or higher dimensions: the limit
\[
 \lim_{t\rightarrow +\infty} \int_{\Rm^d} \frac{|u(x,t)|^2}{|x|^2} dx = 0
\]
holds for all free waves $u$. In general this is not ture in the $2$-dimensional case. Since Hardy's inequality does not hold, we do not even know whether the integral 
\[
 \int_{\Rm^2} \frac{|u(x,t)|^2}{|x|^2} dx
\]
is finite or not. Nevertheless, a careful review of the proof given in \cite{dkm3} shows that we may substitute the limit above by a weaker version: 

\begin{lemma}
 Let $u$ be a solution to $2$-dimensional free wave equation $\partial_t^2 u - \Delta u = 0$ with smooth and compactly supported initial data. Then we have
 \[
  \lim_{t\rightarrow +\infty} \int_{|x|>t-\eta} \frac{|u(x,t)|^2}{|x|^2} dx = 0, \quad \forall \eta \in \Rm.
 \]
\end{lemma}
\begin{proof}
 This immediately follows the well-known dispersive estimate: if the initial data is smooth and compactly supported, then we have $|u(x,t)|\lesssim |t|^{-1/2}$. 
\end{proof}

The following result can be proved by Poisson's formula and smooth approximation techniques. The authors would like to mention that although this result seems to be a direct consequence of radiation fields given above, it is actually an ingredient of the proof of radiation fields. 
\begin{proposition} \label{hollow center}
 Let $u$ be a free wave with finite energy. Then we always have 
 \[
  \lim_{\eta\rightarrow +\infty} \sup_{t\geq \eta} \int_{|x|<t-\eta} \left(|\nabla u(x,t)|^2 + |u_t(x,t)|^2\right) dx = 0.
 \]
\end{proposition}

\section{Morawetz Identity and Estimates}

\subsection{Morawetz identity}

\begin{proposition}[Morawetz identity] 
Let $u$ be a solution to (CP1) with a finite energy $E$. Then the following identity holds for any radius $R>0$ and time $t_1<t_2$. Here $\sigma_R$ is the regular line measure of the circle $|x|=R$ in $\Rm^2$. 
\begin{align}
 &\frac{1}{2R} \int_{t_1}^{t_2} \int_{|x|<R} \left(|\nabla u|^2 + |u_t|^2+\frac{p-3}{p+1}|u|^{p+1}\right) dx dt+ \frac{1}{4R^2} \int_{t_1}^{t_2} \int_{|x|=R} |u|^2 d\sigma_R(x) dt\nonumber \\
 & + \int_{t_1}^{t_2} \int_{|x|>R} \left(\frac{|\slashed{\nabla} u|^2}{|x|} + \frac{p-1}{2(p+1)}\cdot \frac{|u|^{p+1}}{|x|} - \frac{1}{4}\cdot\frac{|u|^2}{|x|^3}\right) dx dt  \nonumber \\
 & + \! \sum_{i=1,2} \int_{|x|<R} \left. \left(\frac{R^2\!-\!|x|^2}{2R^2} |u_r|^2 \!+\! \frac{1}{2}\left|\frac{|x|}{R}u_r \!+\! \frac{u}{2R} \!+\! (-1)^i  u_t\right|^2 \!+\! \frac{3|u|^2}{8R^2}\!+\!\frac{|\slashed{\nabla}u|^2}{2} \!+\! \frac{|u|^{p+1}}{p+1}\right)\right|_{t=t_i} \!dx \nonumber \\
 & + \sum_{i=1,2} \int_{|x|>R} \left. \left(\frac{1}{2}\left|u_r + \frac{u}{2|x|} +(-1)^i u_t\right|^2+\frac{|\slashed{\nabla}u|^2}{2} + \frac{|u|^{p+1}}{p+1} - \frac{|u|^2}{8|x|^2}\right)\right|_{t=t_i} dx  = 2E. \nonumber
\end{align}
\end{proposition}
\begin{remark} \label{upper bound of u2x2}
 If $u$ is a solution to (CP1) with a finite energy $E$, then we have 
 \begin{align*}
  \int_{|x|>R'} \frac{|u|^2}{|x|^2} dx & \leq \left(\int_{|x|>R'} (|u|^2)^{\frac{p+1}{2}} dx \right)^{\frac{2}{p+1}} \left(\int_{|x|>R'} (|x|^{-2})^\frac{p+1}{p-1} dx \right)^{\frac{p-1}{p+1}} \\
  & \lesssim_p  (R')^{-\frac{4}{p+1}} \left(\int_{|x|>R'} |u|^{p+1} dx \right)^{\frac{2}{p+1}}\\
  & \lesssim_p (R')^{-\frac{4}{p+1}} E^{\frac{2}{p+1}}.
 \end{align*}
 Thus the integrals in the inequality above are all finite.  
\end{remark}
\begin{proof}
We follow a similar argument to the one given by Perthame and Vega in the final section of their work \cite{benoit}. Some terms in 2-dimensional case come with a different sign from the higher dimensional case thus we have to work more carefully. Let us first consider solutions with compact support. We will calculate as though the solutions are sufficiently smooth, otherwise smooth approximation techniques can be applied. Given a positive constant $R$, we define two radial functions $\Psi$ and $\varphi$ by
 \begin{align*}
  &\nabla \Psi = \left\{\begin{array}{ll} x, & \hbox{if} \; |x|\leq R; \\ Rx/|x|, & \hbox{if}\; |x|\geq R;\end{array}\right.&
  &\varphi = \left\{\begin{array}{ll} 1/2, & \hbox{if} \; |x|\leq R; \\ 0, & \hbox{if}\; |x|> R.\end{array}\right.&
 \end{align*}
 Since $u$ is defined for all time $t \in \Rm$, we may also define a function on $\Rm$
 \[
  \mathcal{E}(t) = \int_{\Rm^2}  u_t(x,t) \left(\nabla u(x,t)\cdot \nabla \Psi + u(x,t)\left(\frac{\Delta \Psi}{2} - \varphi\right)\right) dx.
 \]
 We may differentiate $\mathcal{E}$, utilize the equation $u_{tt} - \Delta u =  -|u|^{p-1}u$, apply integration by parts and obtain
 \begin{align*}
   -\mathcal{E}'(t) & = \int_{\Rm^2} \left(\sum_{i,j=1}^2 u_i \Psi_{ij} u_j-\varphi |\nabla u|^2 + \varphi |u_t|^2 \right) dx + \frac{1}{4} \int_{\Rm^2} \nabla (|u|^2)\cdot  \nabla \left(\Delta \Psi - 2\varphi\right) dx\\
   & \qquad \qquad + \int_{\Rm^2} |u|^{p+1}\left(\frac{p-1}{2(p+1)}\Delta \Psi - \varphi\right) dx = I_1 + I_2 + I_3.
 \end{align*}
 Here we have
 \begin{align*}
  &\Psi_{ij} = \left\{\begin{array}{ll} \delta_{ij}, & \hbox{if} \; |x|< R; \\ \frac{R\delta_{ij}}{|x|} - \frac{R x_i x_j}{|x|^3}, & \hbox{if}\; |x|> R;\end{array}\right.&
  &\Delta \Psi = \left\{\begin{array}{ll} 2, & \hbox{if} \; |x|< R; \\ R/|x|, & \hbox{if}\; |x|> R;\end{array}\right.&
 \end{align*}
 \begin{align*}
  \Delta \Psi - 2\varphi = \left\{\begin{array}{ll} 1, & \hbox{if} \; |x|\leq R; \\ R/|x|, & \hbox{if}\; |x|\geq R;\end{array}\right. \in C(\Rm^2).
 \end{align*}
when $|x|>R$, we may calculate
\begin{align*}
 \sum_{i,j=1}^2 u_i \Psi_{ij} u_j = \sum_{i,j=1}^2 u_i \left(\frac{R\delta_{ij}}{|x|} - \frac{R x_i x_j}{|x|^3}\right) u_j
= \frac{R}{|x|}|\nabla u|^2 - \frac{R|\nabla u \cdot x|^2}{|x|^3}
= \frac{R}{|x|}|\slashed{\nabla} u|^2.
\end{align*} 
Thus we have 
\begin{align}
 I_1 =  \frac{1}{2} \int_{|x|<R} \!\!\left(|\nabla u|^2 + |u_t|^2\right) dx + R\int_{|x|>R} \frac{|\slashed{\nabla} u|^2}{|x|} dx. \label{Morawetz contribution 1}
\end{align}
A basic computation shows
\begin{equation}
 I_3 = \frac{p-3}{2(p+1)}\int_{|x|<R} |u|^{p+1} dx  + \frac{(p-1)R}{2(p+1)} \int_{|x|>R} \frac{|u|^{p+1}}{|x|} dx. \label{Morawetz contribution 2}
\end{equation}
Finally let us calculate $I_2$ carefully 
\begin{align}
 I_2 & = \frac{1}{4}\int_{\Rm^2} \nabla (|u|^2) \cdot \nabla \left(\Delta \Psi - 2\varphi\right) dx \nonumber\\
 & = \frac{1}{4}\int_{|x|>R} \nabla (|u|^2) \cdot \frac{-Rx}{|x|^3} dx \nonumber\\
 & = \frac{1}{4} \int_{|x|>R} \left[\hbox{div}\left(|u|^2 \cdot \frac{-Rx}{|x|^3}\right) - \frac{R}{|x|^3} |u|^2\right]dx \nonumber\\
 & = \frac{1}{4R} \int_{|x|=R} |u|^2 d\sigma_R(x) -\frac{R}{4} \int_{|x|>R} \frac{|u|^2}{|x|^3} dx. \label{Morawetz contribution 3}
\end{align}
Since $-\mathcal{E}'(t) = I_1 + I_2 + I_3$, we have 
\begin{equation}
 \int_{t_1}^{t_2} (I_1+I_2+I_3) dt = \mathcal{E}(t_1) - \mathcal{E}(t_2). \label{Morawetz integral identity}
\end{equation}
We may rewrite $\mathcal{E}(t_1)$ in the form of
\begin{align*}
 R \mathcal{E}(t_1) &= \frac{1}{2}\int_{\Rm^2} \left( R^2|u_t(x,t_1)|^2 +  \left|\nabla u(x,t_1)\cdot \nabla \Psi + u(x,t_1)\left(\frac{\Delta \Psi}{2} - \varphi\right)\right|^2 \right) dx\\
 & \qquad - \frac{1}{2} \int_{\Rm^2} \left|\nabla u(x,t_1)\cdot \nabla \Psi + u(x,t_1)\left(\frac{\Delta \Psi}{2} - \varphi\right) - R u_t(x,t_1)\right|^2 dx\\
 & = J_1 - J_2
\end{align*}
We then calculate $J_1, J_2$ carefully 
\begin{align*}
 J_1 & = \frac{1}{2}\int_{\Rm^2} \left(R^2 |u_t|^2 + |\nabla u\cdot \nabla \Psi|^2 + \left(\frac{\Delta \Psi}{2} - \varphi\right) \nabla (|u|^2) \cdot \nabla \Psi + \left(\frac{\Delta \Psi}{2} - \varphi\right)^2 |u|^2 \right) dx\\
 & = \frac{1}{2} \int_{\Rm^2} \left[R^2 |u_t|^2 + |\nabla u\cdot \nabla \Psi|^2 - \hbox{div} \left(\left(\frac{\Delta \Psi}{2} - \varphi\right)\nabla \Psi \right) |u|^2 + \left(\frac{\Delta \Psi}{2} - \varphi\right)^2 |u|^2\right] dx
\end{align*}
A basic calculation shows 
\[
 \hbox{div} \left(\left(\frac{\Delta \Psi}{2} - \varphi\right)\nabla \Psi \right) = 
 \left\{\begin{array}{ll} 1, & \hbox{if} \; |x|< R; \\ 0, & \hbox{if}\; |x|> R;\end{array}\right. 
\]
Thus we have 
\begin{align*}
 J_1  = & \frac{1}{2} \int_{|x|<R} \left[R^2 |u_t|^2 + |x\cdot \nabla u|^2 -\frac{3}{4} |u|^2\right] dx + \frac{1}{2} \int_{|x|>R} \left[R^2 |u_t|^2 + R^2|u_r|^2 + \frac{R^2 |u|^2}{4|x|^2}\right] dx \\
  = & R^2 E  - R^2 \int_{|x|<R} \left[\frac{R^2-|x|^2}{2R^2} |u_r|^2 + \frac{3}{8R^2} |u|^2 + \frac{1}{2}|\slashed{\nabla} u|^2 + \frac{1}{p+1} |u|^{p+1}\right] dx\\
 & + \frac{R^2}{8} \int_{|x|>R} \frac{|u|^2}{|x|^2} dx - R^2 \int_{|x|>R} \left(\frac{1}{2}|\slashed{\nabla} u|^2 + \frac{1}{p+1} |u|^{p+1}\right) dx
\end{align*}
In addition we have 
\[
 J_2 =  \frac{1}{2} \int_{|x|<R} \left|x\cdot \nabla u + \frac{1}{2} u - R u_t\right|^2 dx + \frac{R^2}{2} \int_{|x|>R} \left|\frac{x}{|x|} \cdot\nabla u + \frac{u}{2|x|} - u_t\right|^2 dx. 
\]
Combining $J_1, J_2$, we obtain 
\begin{align*}
 R \mathcal{E}(t_1) = & R^2 E  - R^2 \int_{|x|>R} \left(\frac{1}{2} \left|u_r \!+\! \frac{u}{2|x|} \!-\! u_t\right|^2 +\frac{|\slashed{\nabla} u|^2}{2} + \frac{|u|^{p+1}}{p+1} - \frac{|u|^2}{8|x|^2}\right) dx\\
 & -\! R^2 \int_{|x|<R} \!\left[\frac{R^2\!-\! |x|^2}{2R^2} |u_r|^2 \!+\! \frac{1}{2}\left|\frac{|x|}{R} u_r \!+\! \frac{u}{2R} \!-\! u_t\right|^2 \!+\! \frac{|\slashed{\nabla} u|^2}{2} \!+\! \frac{|u|^{p+1}}{p+1} \!+\!\frac{3|u|^2}{8R^2} \right]\! dx.
\end{align*}
Finally we may find a similar expression of $-R \mathcal{E}(t_2)$ 
\begin{align*}
 -R \mathcal{E}(t_2) = & R^2 E  - R^2 \int_{|x|>R} \left(\frac{1}{2} \left|u_r \!+\! \frac{u}{2|x|} \!+\! u_t\right|^2 +\frac{|\slashed{\nabla} u|^2}{2} + \frac{|u|^{p+1}}{p+1} - \frac{|u|^2}{8|x|^2}\right) dx\\
 & -\! R^2 \int_{|x|<R} \!\left[\frac{R^2\!-\! |x|^2}{2R^2} |u_r|^2 \!+\! \frac{1}{2}\left|\frac{|x|}{R} u_r \!+\! \frac{u}{2R} \!+\! u_t\right|^2 \!+\! \frac{|\slashed{\nabla} u|^2}{2} \!+\! \frac{|u|^{p+1}}{p+1} \!+\!\frac{3|u|^2}{8R^2} \right]\! dx,
\end{align*}
then plug all the expressions of $I_1, I_2, I_3$ and $\mathcal{E}(t_1), \mathcal{E}(t_2)$ into the integral identity \eqref{Morawetz integral identity} to finish the proof if the solution is compactly supported. In order to deal with general solutions $u$, we fix a smooth radial cut-off function $\phi: \Rm^2 \rightarrow [0,1]$ so that 
\[
 \phi(x) = \left\{\begin{array}{ll} 1, & \hbox{if} \; |x|\leq 1; \\ 0, & \hbox{if}\; |x| > 2;\end{array}\right.
\]
define initial data $(u_{0,R'}(x),u_{1,R'}(x)) = \phi(x/R') (u(x,t_1), u_t(x,t_1))$ and consider the corresponding solution $u_{R'}$ to (CP1). The argument above shows that $u_{R'}$ satisfies the Morawetz identity. We observe
\begin{itemize}
 \item The identity $u_{R'}(x,t) = u(x,t)$ holds if $|x|<R'+t_1-t$ by finite speed of propagation;
 \item $E(u_{0,R'}, u_{1,R'}) \rightarrow E$ as $R' \rightarrow \infty$. The integral of $|u(x,t_1)|^2/|x|^2$ can be dealt with by Remark \ref{upper bound of u2x2};
 \item The energies of $u_{R'}$ and $u$ in the region where $u_{R'} \neq u$ both converge  to zero as $R'\rightarrow +\infty$ by finite speed of propagation and energy conservation law. 
 \end{itemize}
These facts enable us to take the limit $R'\rightarrow +\infty$ and prove the Morawetz identity for general solutions $u$ without compact support. 
\end{proof}

\subsection{Morawetz Inequalities}

A combination of Morawetz identity and finite speed of propagation gives a few useful inequalities, which will be my main tool in the rest of this paper. The key observation here is that if $R$ is large, the first term in the Morawetz identity is almost $2E$ when $t_1\leq -R$ and $t_2 \leq R$, thus all other terms must be small.  

\begin{corollary} \label{energy dis}
 Let $u$ be a solution to (CP1) with initial data $(u_0,u_1)\in (\dot{H}^1(\Rm^2) \cap L^{p+1}(\Rm^2))\times L^2(\Rm^2)$. Given any $R>0$, $r\geq 0$, $0\leq \mu_1\leq \frac{2(p-1)}{p+1}$ and $0\leq \mu_2\leq \frac{1}{p+1}$ we have  
\begin{align*}
 \sum_{j=1}^6 M_j \leq & \int_{\Rm^2} \min\{|x|/R, 1\} \left(|\nabla u_0|^2 + |u_1|^2 + \frac{2}{p+1}|u_0|^{p+1}\right) dx\\
 & + \sum_{\pm} \int_{|x|>R} \left.\left(\frac{9|u|^2}{8|x|^2} - \lambda_2 |u|^{p+1}\right)\right|_{t=\pm (R+r)} dx + \int_{-R-r}^{R+r} \int_{|x|>R} \left(\frac{5|u|^2}{4|x|^3}-\frac{\lambda_1 |u|^{p+1}}{|x|}\right)dx dt.
\end{align*}
The notations $M_j$ represent
\begin{align*}
 M_1 & = \frac{1}{2R} \int_{R<|t|<R+r} \int_{|x|<R} \left(|\nabla u|^2 + |u_t|^2+\frac{p-3}{p+1}|u|^{p+1}\right) dx dt;\\
 M_2 & = \frac{p-5}{2(p+1)R} \int_{-R}^R \int_{|x|<R} |u|^{p+1} dx dt;\\
 M_3 & = \frac{1}{4R^2} \int_{-R-r}^{R+r} \int_{|x|=R} |u|^2 d\sigma_R(x)dt;\\
 M_4 & = \int_{-R-r}^{R+r} \int_{|x|>R} \left(\frac{|\slashed{\nabla} u|^2}{|x|} + \frac{\mu_1 |u|^{p+1}}{|x|} + \frac{|u|^2}{|x|^3}\right) dx dt;\\
 M_5 & = \sum_{\pm} \int_{|x|<R} \! \left. \left(\frac{R^2 \!-\! |x|^2}{2R^2} |u_r|^2 \!+\! \frac{1}{2}\left|\frac{|x|}{R}u_r \!+\!  \frac{u}{2R}  \!\pm \! u_t\right|^2 \!+\! \frac{3|u|^2}{8R^2} \!+\! \frac{|\slashed{\nabla}u|^2}{2} \!+\! \frac{|u|^{p+1}}{p+1} \right)\right|_{t=\pm (R+r)} \!\! dx;\\
 M_6 & = \sum_{\pm} \int_{|x|>R} \left.\left(\frac{1}{2}\left|u_r + \frac{u}{2|x|} \pm u_t\right|^2 +\frac{|\slashed{\nabla}u|^2}{2} + \mu_2 |u|^{p+1} + \frac{|u|^2}{|x|^2}\right)\right|_{t=\pm (R+r)} dx.
\end{align*}
The constants $\lambda_i$'s are defined by $\lambda_1 = \frac{2(p-1)}{p+1} - \mu_1$ and $\lambda_2 = \frac{1}{p+1} - \mu_2$. Moreover, if $\lambda_1,\lambda_2>0$, then there exists a constant $C= C(p,\mu_1,\mu_2)$ so that
\begin{equation} \label{energy dis 2}
  \sum_{j=1}^6 M_j \leq \int_{\Rm^2} \min\{|x|/R, 1\} \left(|\nabla u_0|^2 + |u_1|^2 + \frac{2}{p+1}|u_0|^{p+1}\right) dx + C (R+r) R^{-1-\frac{4}{p-1}}. 
\end{equation}
\end{corollary}
\begin{remark}
If $p \geq 5$ is conformal or super-conformal, then all the terms $M_j$ are nonnegative. If $3\leq p <5$, then all $M_j$'s except for $M_2$ are nonnegative. The terms $M_5$ and $M_6$ are the most relevant to the energy distribution at large time and will be used in later sections. Some of $M_j$'s will not be used in this work but we still give them here for completeness. 
\end{remark}
\begin{proof}
We first choose $t_1=-R-r$, $t_2=R+r$ in the Morawetz identity, add
\[ 
  \int_{-R-r}^{R+r} \int_{|x|>R} \left(\frac{5|u|^2}{4|x|^3} - \frac{\lambda_1 |u|^{p+1}}{|x|}\right) dx dt + \sum_{i=1,2} \int_{|x|>R} \left(\frac{9|u(x,t_i)|^2}{8|x|^2} - \lambda_2 |u(x,t_i)|^{p+1}\right)  dx 
\]
in both sides (here $\lambda_1 = \frac{2(p-1)}{p+1} - \mu_1\geq 0$, $\lambda_2 = \frac{1}{p+1} - \mu_2 \geq  0$), and obtain 
\begin{align*}
 &\frac{1}{2R} \int_{-R-r}^{R+r} \int_{|x|<R} \left(|\nabla u|^2 + |u_t|^2+\frac{p-3}{p+1}|u|^{p+1}\right) dx + \sum_{j=3}^6 M_j \\
 & = 2E + \sum_{\pm} \int_{|x|>R} \left.\left(\frac{9|u|^2}{8|x|^2} - \lambda_2 |u|^{p+1}\right)\right|_{t=\pm (R+r)} dx + \int_{-R-r}^{R+r} \int_{|x|>R} \left(\frac{5|u|^2}{4|x|^3}-\frac{\lambda_1 |u|^{p+1}}{|x|}\right)dx. 
\end{align*}
The first term above can be written as a sum of three terms
\begin{align*}
& \frac{1}{2R} \int_{-R-r}^{R+r} \int_{|x|<R} \left(|\nabla u|^2 + |u_t|^2+\frac{p-3}{p+1}|u|^{p+1}\right) dx\\
 = & M_1 + \frac{1}{2R} \int_{-R}^{R} \int_{|x|<R} \left(|\nabla u|^2 + |u_t|^2+\frac{p-3}{p+1}|u|^{p+1}\right) dx\\
 = & M_1 + M_2 + \frac{1}{2R} \int_{-R}^{R} \int_{|x|<R} \left(|\nabla u|^2 + |u_t|^2+\frac{2}{p+1}|u|^{p+1}\right) dx
\end{align*}
Thus
\begin{align}
 & \frac{1}{2R} \int_{-R}^{R}  \int_{|x|<R} \left(|\nabla u|^2 + |u_t|^2+\frac{2}{p+1}|u|^{p+1}\right) dx dt + \sum_{j=1}^6 M_j \label{energy dis mid}\\
 & = 2E + \sum_{\pm} \int_{|x|>R} \left.\left(\frac{9|u|^2}{8|x|^2} - \lambda_2 |u|^{p+1}\right)\right|_{t=\pm (R+r)} dx + \int_{-R-r}^{R+r} \int_{|x|>R} \left(\frac{5|u|^2}{4|x|^3}-\frac{\lambda_1 |u|^{p+1}}{|x|}\right)dx. \nonumber
\end{align}
In order to prove the first inequality we only need to show
\begin{align*}
 I = 2E - \frac{1}{2R} \int_{-R}^{R} & \int_{|x|<R} \left(|\nabla u|^2 + |u_t|^2+\frac{2}{p+1}|u|^{p+1}\right) dx dt\\
  & \leq \int_{\Rm^2} \min\{|x|/R, 1\} \left(|\nabla u_0|^2 + |u_1|^2 + \frac{2}{p+1}|u_0|^{p+1}\right) dx.
\end{align*}
This immediately follows energy conservation law and finite speed of propagation of energy
 \begin{align*}
 I =  & \frac{1}{R}\int_{-R}^R \int_{|x|>R} \left(\frac{1}{2}|\nabla u|^2 + \frac{1}{2}|u_t|^2 + \frac{1}{p+1}|u|^{p+1}\right) dx dt \\
  \leq & \frac{1}{R} \int_{-R}^R \int_{|x|>R-|t|} \left(\frac{1}{2}|\nabla u_0|^2 + \frac{1}{2}|u_1|^2 + \frac{1}{p+1}|u_0|^{p+1}\right) dx dt\\
  = & \frac{1}{R} \int_{\Rm^2} \min\{|x|, R\} \left(|\nabla u_0|^2 + |u_1|^2 + \frac{2}{p+1}|u_0|^{p+1}\right) dx.
 \end{align*}
Finally if $\lambda_1, \lambda_2 > 0$, we need to find an upper bound of the integrals
\[
 \sum_{\pm} \int_{|x|>R} \left.\left(\frac{9|u|^2}{8|x|^2} - \lambda_2 |u|^{p+1}\right)\right|_{t=\pm (R+r)} dx + \int_{-R-r}^{R+r} \int_{|x|>R} \left(\frac{5|u|^2}{4|x|^3}-\frac{\lambda_1 |u|^{p+1}}{|x|}\right)dx.
\]
If $u \gg |x|^{-2/(p-1)}$, then $ |u|^2/|x|^2 \ll |u|^{p+1}$, Thus we always have 
\begin{align*}
 &\frac{9|u|^2}{8|x|^2} - \frac{\lambda_2 |u|^{p+1}}{p+1} \lesssim_{p,\mu_2} |x|^{-2-\frac{4}{p-1}};& &\frac{5|u|^2}{4|x|^3} -  \frac{\lambda_1 |u|^{p+1}}{|x|} \lesssim_{p,\mu_1} |x|^{-3-\frac{4}{p-1}}.&
\end{align*}
Plugging these upper bounds in the integrals above, we obtain 
\begin{align*}
 \sum_{\pm} \int_{|x|>R} \left.\left(\frac{|9u|^2}{8|x|^2} - \lambda_2 |u|^{p+1}\right)\right|_{t=\pm (R+r)} dx + \int_{-R-r}^{R+r} \int_{|x|>R} & \left(\frac{5|u|^2}{4|x|^3}-\frac{\lambda_1 |u|^{p+1}}{|x|}\right)dx\\
 & \lesssim_{p,\mu_1,\mu_2} (R+r) R^{-1-\frac{4}{p-1}}.
\end{align*}
This immediately proves inequality \eqref{energy dis 2} and finishes the proof.
\end{proof}

\section{Energy Distribution}

In this section we prove Theorem \ref{main 1}. It suffices to consider the positive time direction $t>0$, since the wave equation is time-reversible. We start with the conformal and super-conformal case $p\geq 5$. 

\subsection{Conformal and Super-conformal Case}

The proof mainly relies on 
\begin{lemma} \label{lemma p 5}
 Assume $p\geq 5$. Let $u$ be a finite-energy to (CP1). Then the following inequalities hold for large time $t>0$.
\begin{align*}
 \int_{|x|<t} & \frac{t-|x|}{t} e(x,t) dx + \int_{\Rm^2} \left(|u_r+u_t|^2 + |\slashed{\nabla} u|^2 + |u|^{p+1}\right) dx\\
 & \lesssim_p \int_{\Rm^2}\min\{|x|/t, 1\} e(x,0) dx
  + \sum_{\pm} \int_{|x|>t} \frac{|u(x,\pm t)|^2}{|x|^2} dx + \int_{-t}^{t} \int_{|x|>t} \frac{|u(x,t')|^2}{|x|^3} dx dt';\\
\int_{|x|<t} & \frac{t-|x|}{t} e(x,t) dx + \int_{\Rm^2} \left(|u_r+u_t|^2 + |\slashed{\nabla} u|^2 + |u|^{p+1}\right) dx \lesssim_p \int_{\Rm^2}\min\{|x|/t, 1\} e(x,0) dx + t^{-\frac{4}{p-1}}.
\end{align*} 
\end{lemma}
\begin{proof}
 We choose $R=t$, $r= 0$, $\mu_1 = \frac{2(p-1)}{p+1}$ and $\mu_2 = \frac{1}{p+1}$ in Corollary \ref{energy dis}, then we have (the two integrals below are evaluated at time $t$)
\begin{align*}
  \int_{|x|<t} & \left(\frac{t^2 \!-\! |x|^2}{2t^2} |u_r|^2 \!+\! \frac{1}{2}\left|\frac{|x|}{t}u_r \!+\!  \frac{u}{2t}  \!+ \! u_t\right|^2 \!+\! \frac{3|u|^2}{8t^2} \!+\! \frac{|\slashed{\nabla}u|^2}{2} \!+\! \frac{|u|^{p+1}}{p+1} \right) dx \\
  & + \int_{|x|>t} \left(\frac{1}{2}\left|u_r + \frac{u}{2|x|} + u_t\right|^2 +\frac{|\slashed{\nabla}u|^2}{2} + \frac{1}{p+1} |u|^{p+1} + \frac{|u|^2}{|x|^2}\right) dx \lesssim_p \hbox{RH1}.
\end{align*}
Here RH1 represents the right hand side of the first inequality in Lemma \ref{lemma p 5}. Thus we may verify the first inequality by observing
\[
 |\nabla u|^2 + |u_t|^2 \lesssim_1 |u_r|^2 + \left|\frac{|x|}{t}u_r \!+\!  \frac{u}{2t}  \! + \! u_t\right|^2 + \frac{|u|^2}{t^2} + |\slashed{\nabla} u|^2, \quad \hbox{if} \; |x|<t;
\]
and 
\begin{align*}
 |u_r+u_t|^2 & \lesssim_1 \frac{(t-|x|)^2}{t^2} |u_r|^2 + \left|\frac{|x|}{t}u_r \!+\!  \frac{u}{2t}  \! + \! u_t\right|^2 + \frac{|u|^2}{t^2}, & &\hbox{if}\; |x|<t;\\
 |u_r+u_t|^2 & \lesssim_1 \left|u_r + \frac{u}{2|x|} + u_t\right|^2 + \frac{|u|^2}{|x|^2}, & & \hbox{if}\; |x|>t.
\end{align*}
The second inequality can be proved in the same way by choosing $R=t$, $r= 0$, $\mu_1 = \frac{p-1}{p+1}$, $\mu_2 = \frac{1}{2(p+1)}$ in Corollary \ref{energy dis} and using inequality \eqref{energy dis 2}. 
\end{proof}

\paragraph{Proof of Part (a)(b)} It is clear that Part (a) and (b) immediately follows the second inequality given in Lemma \ref{lemma p 5}. We only need to let $t\rightarrow +\infty$ and apply the dominated convergence theorem.  

\paragraph{Proof of part (c), smaller decay rate} If $\kappa \in (0,\frac{4}{p-1})$ and $E_\kappa (u_0,u_1)< +\infty$ , then we may apply Lemma \ref{lemma p 5} again
\begin{align*}
  t^\kappa \int_{|x|<t} \frac{t-|x|}{t} e(x,t) dx + & t^\kappa \int_{\Rm^2}\left(|u_r+u_t|^2 + |\slashed{\nabla} u|^2 + |u|^{p+1}\right) dx \\
 & \lesssim_{p} \int_{\Rm^2}\min\{|x|t^{\kappa-1}, t^\kappa\} e(x,0) dx + t^{\kappa-\frac{4}{p-1}}.
\end{align*}
We observe the facts 
\begin{align*}
 &\min\{|x|t^{\kappa-1}, t^\kappa\} \leq |x|^\kappa,& &\lim_{t\rightarrow +\infty} |x|t^{\kappa-1} = 0,&
\end{align*}
then apply the dominated convergence theorem again and conclude that the upper bound found above converges to zero as $t\rightarrow +\infty$. This finishes the proof of part (c).
 
\begin{remark} We may slightly improve the decay estimate if decay rate $\kappa \in (0,\frac{4}{p-1})$. By Lemma \ref{lemma p 5} we have
\begin{align*}
 \int_1^\infty t^{\kappa-1} & \left(\int_{|x|<t} \frac{t-|x|}{t} e(x,t) dx + \int_{\Rm^2}\left(|u_r+u_t|^2 + |\slashed{\nabla} u|^2 + |u|^{p+1}\right) dx\right) dt\\
 & \lesssim_p \int_1^\infty \left(\int_{\Rm^2}\min\{|x|t^{\kappa-2}, t^{\kappa-1}\} e(x,0) dx + t^{\kappa-\frac{4}{p-1}-1} \right) dt\\
 & \lesssim_{p, \kappa} \int_{\Rm^2} |x|^\kappa e(x,0) dx + 1 < +\infty.
\end{align*}
\end{remark} 
 
\paragraph{Proof of Part (c), higher decay rate} Now we assume $\kappa \in [\frac{4}{p-1}, 1)$ and $E_\kappa (u_0,u_1) < +\infty$. We utilize the first inequality given in Lemma \ref{lemma p 5} and obtain 
\begin{align}
 \int_{|x|<t} & \frac{t-|x|}{t} e(x,t) dx + \int_{\Rm^2}\left(|u_r+u_t|^2 + |\slashed{\nabla} u|^2 + |u|^{p+1}\right) dx \nonumber\\
 & \lesssim_p \int_{\Rm^2}\min\{|x|/t, 1\} e(x,0) dx
  + \sum_{\pm} \int_{|x|>t} \frac{|u(x,\pm t)|^2}{|x|^2} dx + \int_{-t}^{t} \int_{|x|>t} \frac{|u(x,t')|^2}{|x|^3} dx dt' \label{part d inequality}
\end{align}
We have already known how to deal with the first integral in the right hand side above. Thus we still need to deal with the last two integrals. The upper bounds of these two integrals can be found by a combination of finite speed of propagation and a weighted Hardy inequality. We first prove two lemmata. 
\begin{lemma}[Weighted Hardy inequality]
Let $\kappa\in (0,1)$. If $v \in \dot{H}^1(\Rm^2)$ and $0\leq R_1\leq R$, then we have 
 \[ 
  \int_{|x|>R} \frac{|v(x)|^2}{|x|^2} dx \lesssim_\kappa \int_{|x|>R} (R^{-\kappa} - |x|^{-\kappa}) (|x|-R_1)^\kappa |\nabla v(x)|^2 dx.
 \]
\end{lemma}
\begin{proof}
 Without loss of generality we may calculate as though $v$ is smooth. We have
 \begin{align*}
  \int_{|x|>R} \frac{|v(x)|^2}{|x|^2} dx &= \int_R^\infty \int_0^{2\pi} r^{-1} |v(r,\theta)|^2 d\theta dr\\
  & \leq \int_R^\infty \int_0^{2\pi} r^{-1} \left(\int_r^\infty |v_r(r',\theta)| dr' \right)^2 d\theta dr\\
  & \leq \int_R^\infty \int_0^{2\pi} r^{-1} \left(\int_r^\infty (r'-R_1)^\kappa r' |v_r(r',\theta)|^2 dr'\right)\left(\int_r^\infty (r'-R_1)^{-\kappa} (r')^{-1} dr'\right) d\theta dr\\
  & \lesssim_\kappa \int_R^\infty \int_0^{2\pi} r^{-1-\kappa} \left(\int_r^\infty (r'-R_1)^\kappa r' |v_r(r',\theta)|^2 dr'\right) d\theta dr\\
  & \lesssim_\kappa \int_R^\infty r^{-1-\kappa} \left(\int_{|x|>r} (|x|-R_1)^\kappa |\nabla v(x)|^2 dx\right) dr\\
  & \lesssim_\kappa \int_{|x|>R} (R^{-\kappa} - |x|^{-\kappa}) (|x|-R_1)^\kappa |\nabla v(x)|^2 dx.
 \end{align*}
\end{proof}
\begin{lemma}[Finite speed of propagation]
Let $a(r)$ be an absolutely continuous and increasing function defined on $[R,+\infty)$. If $u$ is a finite-energy solution to (CP1), then for any time $t' \in [-R,R]$ we have 
\[
 \int_{|x|>R} a(|x|) e(x,t') dx \leq \int_{|x|>R-|t'|} a(|x|+|t'|) e(x,0) dx. 
\] 
\end{lemma}
\begin{proof}
 This is a combination of finite speed of propagation and Fubini's theorem. In fact we have
 \begin{align*}
  \int_{|x|>R} a(|x|) e(x,t') dx & = \int_R^\infty a'(r) \left(\int_{|x|>r} e(x,t') dx\right) dr + a(R) \int_{|x|>R} e(x,t') dx\\
  & \leq \int_R^\infty a'(r) \left(\int_{|x|>r-|t'|} e(x,0) dx\right) dr + a(R) \int_{|x|>R-|t'|} e(x,0) dx\\
  & = \int_{|x|>R-|t'|} a(|x|+|t'|) e(x,0) dx. 
 \end{align*}
\end{proof}
\noindent Now we are ready to find the upper bounds of the last two inequalities in \eqref{part d inequality}. We apply the two lemmata above and obtain ($|t'|\leq t$)
\begin{align*}
 \int_{|x|>t} \frac{|u(x,t')|^2}{|x|^2} dx & \lesssim_\kappa \int_{|x|>t} (t^{-\kappa} - |x|^{-\kappa}) (|x|-|t'|)^\kappa |\nabla u(x,t')|^2 dx\\
 & \lesssim_1 \int_{|x|>t} (t^{-\kappa} - |x|^{-\kappa}) (|x|-|t'|)^\kappa e(x,t') dx \\
 & \leq \int_{|x|>t-|t'|} (t^{-\kappa}-(|x|+|t'|)^{-\kappa}) |x|^\kappa e(x,0) dx.
\end{align*}
Thus we have
\begin{align*}
 \sum_{\pm} \int_{|x|>t}  \frac{|u(x,\pm t)|^2}{|x|^2} dx \lesssim_\kappa \int_{\Rm^2} (t^{-\kappa}-(|x|+t)^{-\kappa}) |x|^\kappa e(x,0) dx;
\end{align*} 
and 
\begin{align*}
  \int_{-t}^{t} \int_{|x|>t} \frac{|u(x,t')|^2}{|x|^3} dx dt' &\leq t^{-1} \int_{-t}^{t} \int_{|x|>t} \frac{|u(x,t')|^2}{|x|^2} dx dt'\\
  & \lesssim_\kappa t^{-1} \int_{-t}^t \left(\int_{|x|>t-|t'|} (t^{-\kappa}-(|x|+|t'|)^{-\kappa}) |x|^\kappa e(x,0) dx\right) dt'\\
  & \lesssim_1 \int_{\Rm^2} (t^{-\kappa}-(|x|+t)^{-\kappa}) |x|^\kappa e(x,0) dx.
\end{align*}
Plugging these two upper bounds in \eqref{part d inequality}, we obtain
\begin{align*}
  \int_{|x|<t} \frac{t-|x|}{t} e(x,t) dx & + \int_{\Rm^2}\left(|u_r+u_t|^2 + |\slashed{\nabla} u|^2 + |u|^{p+1}\right) dx\\
 & \lesssim_{p, \kappa} \int_{\Rm^2} \left[\min\{|x|/t, 1\} + (t^{-\kappa}-(|x|+t)^{-\kappa}) |x|^\kappa\right] e(x,0) dx.
\end{align*}
Therefore
\begin{align*}
  t^\kappa \int_{|x|<t} \frac{t-|x|}{t} e(x,t) dx & + t^\kappa \int_{\Rm^2}\left(|u_r+u_t|^2 + |\slashed{\nabla} u|^2 + |u|^{p+1}\right) dx\\
 & \lesssim_{p, \kappa} \int_{\Rm^2} \left[\min\{|x|t^{\kappa-1}, t^\kappa\} + \left(1-\left(\frac{t}{t+|x|}\right)^{\kappa}\right) |x|^\kappa\right] e(x,0) dx.
\end{align*}
Finally we may complete the proof by dominated convergence theorem.

\subsection{Sub-conformal Case}

In this subsection, we consider the case $3<p<5$. In general, the argument is almost the same as in the super-conformal case. The only difference is the presence of an additional term in the upper bounds: 
\[
 \frac{5-p}{2(p+1)R} \int_{-R}^R \int_{|x|<R} |u|^{p+1} dx dt
\]
Thus we have to find a good upper bound of it first. We choose $R=t$, $r=0$ in Corollary \ref{energy dis}, follow the same argument as in the case $p\geq 5$ and obtain
\begin{align*}
 & \mu_2 \int_{\Rm^2}  (|u(x,t)|^{p+1} +|u(x,-t)|^{p+1}) dx + c_1 \sum_{\pm} \int_{|x|<t} \frac{t-|x|}{t}\left(|\nabla u(x,\pm t)|^2 + |u_t(x,\pm t)|^2 \right) dx  \\
  & + c_2 \sum_{\pm} \int_{\Rm^2} \left(|(u_r\pm u_t)(x,\pm t)|^2 + |\slashed{\nabla} u(x,\pm t)|^2 \right) dx \\
 &\quad \leq  \frac{5-p}{2(p+1)t} \int_{-t}^t \int_{|x|<t} |u(x,t')|^{p+1} dx dt' + 2\int_{\Rm^2} \min\{|x|/t, 1\} e(x,0) dx + C_{p, \mu_2} t^{-\frac{4}{p-1}}.
\end{align*}
Here we choose a positive constant $\mu_2$ slightly smaller than $\frac{1}{p+1}$. The constant $c_1,c_2>0$ are absolute constants. For convenience we introduce the notation 
\begin{align*}
 Q(t) = \mu_2 \int_{\Rm^2}  (|u(x,t)|^{p+1} +|u(x,-t)|^{p+1}) dx & + c_1 \sum_{\pm} \int_{|x|<t} \frac{t-|x|}{t}\left(|\nabla u(x,\pm t)|^2 + |u_t(x,\pm t)|^2 \right) dx\\
 & +c_2\sum_{\pm} \int_{\Rm^2} \left(|(u_r\pm u_t)(x,\pm t)|^2 + |\slashed{\nabla} u(x,\pm t)|^2 \right) dx.
\end{align*}
The inequality above implies that $Q(t)$ satisfies the recurrence formula 
\[
 Q(t) \leq \frac{\lambda}{t} \int_{0}^t Q(t') dt' + 2\int_{\Rm^2} \min\{|x|/t, 1\} e(x,0) dx + C_{p, \mu_2} t^{-\frac{4}{p-1}}.
\]
Here 
\[
 \lambda = \frac{5-p}{2(p+1)\mu_2} \approx \frac{5-p}{2}< 1.
\]
\paragraph{Proof of part (a)(b)} We may rewrite the recurrence formula as 
\[
 Q(t) \leq \frac{\lambda}{t} \int_{0}^t Q(t') dt' + o(1),
\]
We may take upper limits of both sides and obtain an inequality
\[
 \limsup_{t\rightarrow +\infty} Q(t) \leq \limsup_{t\rightarrow +\infty} \frac{\lambda}{t} \int_{0}^t Q(t') dt' \leq \lambda \limsup_{t\rightarrow +\infty} Q(t).
\]
We recall the fact $\lambda \in (0,1)$ and observe that $Q(t) \lesssim E$ is uniformly bounded, therefore we have 
\[
 \limsup_{t\rightarrow +\infty} Q(t) = 0.
\]
This verifies (a)(b).

\paragraph{Proof of part (c)} Now we assume that the initial data satisfy additional decay assumption $E_\kappa(u_0,u_1)<+\infty$ and prove the decay estimates in part (c). We start by multiplying both sides of the recurrence formula by $t^{\kappa-1}$ and integrate from $t=1$ to $t=T$, utilize our assumption on initial data, then obtain 
\begin{align*}
 \int_1^T t^{\kappa-1} Q(t) dt & \leq  \int_1^T t^{\kappa-1}\left(\frac{\lambda}{t}\int_0^t Q(t') dt'\right) dt + C_\kappa \int_{\Rm^2} \min\{|x|, |x|^\kappa\}e(x,0) dx + C_{p, \mu_2,\kappa}\\
 & \leq \frac{\lambda}{1-\kappa}\int_0^T \min\{(t')^{\kappa-1},1\} Q(t') dt' + C_\kappa \int_{\Rm^2} \min\{|x|, |x|^\kappa\}e(x,0) dx + C_{p, \mu_2,\kappa}\\
 & \leq \frac{\lambda}{1-\kappa}\int_1^T (t')^{\kappa-1} Q(t') dt' + C_\kappa \int_{\Rm^2} \min\{|x|, |x|^\kappa\}e(x,0) dx + C_\kappa E+C_{p, \mu_2,\kappa}.
\end{align*}
We may choose $\mu_2$ sufficiently close to $\frac{1}{p+1}$  so that the constant 
\[
 \frac{\lambda}{1-\kappa} = \frac{5-p}{2(p+1)\mu_2(1-\kappa)} < 1, 
\]
since we have assumed $\kappa < \frac{p-3}{2}$. Therefore we have 
\begin{align*}
  \int_1^T t^{\kappa-1} Q(t) dt \lesssim_{p,\mu_2,\kappa}  \int_{\Rm^2} \min\{|x|, |x|^\kappa\} e(x,0) dx + E + 1.
\end{align*}
Because neither the right hand side nor the implicit constant here depends on $T$, we may make $T\rightarrow +\infty$ to conclude
\begin{equation} \label{integrability of t kappa Q}
 \int_1^{\infty} t^{\kappa-1} Q(t)  dt < +\infty. 
\end{equation} 
Combining this with the fact $Q(t) \lesssim E$, we have 
\[
 \int_0^{\infty} t^{\kappa-1} Q(t)  dt < +\infty.
\]
We may multiply both sides of the recurrence formula by $t^\kappa$:
\begin{align*}
t^\kappa Q(t) \leq \lambda \int_{0}^t t^{\kappa-1} Q(t') dt' + 2\int_{\Rm^2} \min\{|x|t^{\kappa-1}, t^\kappa\} e(x,0) dx + C_{p, \mu_2} t^{\kappa-\frac{4}{p-1}}.
\end{align*}
Finally we apply dominated convergence theorem to finish the proof. 

\section{Scattering Theory of Non-radial Solutions}

In this section we prove Theorem \ref{main 3}. First of all, energy conservation law gives
 \[
  \left\|\mathbf{S}_L (-t_1) \begin{pmatrix} u(\cdot,t_1)\\ u_t(\cdot,t_1) \end{pmatrix} - \mathbf{S}_L(-t_2) \begin{pmatrix} u(\cdot,t_2)\\ u_t(\cdot,t_2)\end{pmatrix}\right\|_{\dot{H}^1 \times L^2(\Rm^2)} \lesssim_1 E^{1/2}. 
 \]
In addition, we may apply Theorem \ref{main 1} and energy conservation law to obtain 
\[
 \int_{\Rm^2} |u(x,t)|^{p+1} dx \lesssim \max\{1, |t|^{-\kappa}\}.
\]
This implies that $u \in L^q L^{p+1} (\Rm \times \Rm^2)$ for all $q > (p+1)/\kappa$. Since we have assumed $\kappa > \frac{3p+5}{4p}$, thus $u \in L^{\frac{4p(p+1)}{3p+5}} L^{p+1} (\Rm \times \Rm^2)$. As a result, we may choose $s' =  \frac{p+7}{4(p+1)} \in (\frac{1}{4}, \frac{1}{2})$, apply Strichartz estimates and obtain
\begin{align*}
 & \limsup_{t_1,t_2\rightarrow +\infty} \left\|\mathbf{S}_L (-t_1) \begin{pmatrix} u(\cdot,t_1)\\ u_t(\cdot,t_1) \end{pmatrix} - \mathbf{S}_L(-t_2) \begin{pmatrix} u(\cdot,t_2)\\ u_t(\cdot,t_2)\end{pmatrix}\right\|_{\dot{H}^{s'} \times \dot{H}^{s'-1} (\Rm^2)}\\
 = & \limsup_{t_1,t_2\rightarrow +\infty} \left\|\mathbf{S}_L (t_2-t_1) \begin{pmatrix} u(\cdot,t_1)\\ u_t(\cdot,t_1) \end{pmatrix} - \begin{pmatrix} u(\cdot,t_2)\\ u_t(\cdot,t_2)\end{pmatrix}\right\|_{\dot{H}^{s'} \times \dot{H}^{s'-1} (\Rm^2)}\\
 \lesssim & \limsup_{t_1,t_2\rightarrow +\infty} \left\|-|u|^{p-1} u\right\|_{L^{\frac{4(p+1)}{3p+5}} L^{\frac{p+1}{p}}([t_1,t_2]\times \Rm^2)}\\
 = & \limsup_{t_1,t_2\rightarrow +\infty} \|u\|_{L^{\frac{4p(p+1)}{3p+5}} L^{p+1}([t_1,t_2]\times \Rm^2)}^p =0.
\end{align*}
An interpolation between the spaces $\dot{H}^1 \times L^2$ and $\dot{H}^{s'}\times \dot{H}^{s'-1}$ then gives ($s' < 1/2<s_p <1$)
\[ 
 \limsup_{t_1,t_2\rightarrow +\infty} \left\|\mathbf{S}_L (-t_1) \begin{pmatrix} u(\cdot,t_1)\\ u_t(\cdot,t_1) \end{pmatrix} - \mathbf{S}_L(-t_2) \begin{pmatrix} u(\cdot,t_2)\\ u_t(\cdot,t_2)\end{pmatrix}\right\|_{\dot{H}^{s_p} \times \dot{H}^{s_p-1} (\Rm^2)} = 0.
\]
A similar argument shows 
\begin{equation}
 \sup_{t\in \Rm} \left\|\mathbf{S}_L (-t) \begin{pmatrix} u(\cdot,t)\\ u_t(\cdot,t) \end{pmatrix} -  \begin{pmatrix} u_0\\ u_1 \end{pmatrix}\right\|_{\dot{H}^{s_p} \times \dot{H}^{s_p-1} (\Rm^2)} < + \infty. \label{uniform bound sp}
\end{equation}
Next we utilize Sobolev embedding
\begin{align*}
 \|(u_0,u_1)\|_{\dot{H}^{s_p}\times \dot{H}^{s_p-1}} & \lesssim \|(u_0,u_1)\|_{\dot{W}^{1,\frac{2(p-1)}{p+1}}\times L^{\frac{2(p-1)}{p+1}}}\\
 & \lesssim \left\|(1+|x|)^{\kappa/2} (\nabla u_0, u_1)\right\|_{L^2} \left\|(1+|x|)^{-\kappa/2}\right\|_{L^{p-1}}\\
 & \lesssim E_\kappa (u_0,u_1)^{1/2} \left(\int_{\Rm^2}(1+|x|)^{-\frac{(p-1)\kappa}{2}} dx\right)^{1/(p-1)} < +\infty.
\end{align*}
Here $\frac{(p-1)\kappa}{2} > \frac{(p-1)(3p+5)}{8p} > 2$. Combining this with \eqref{uniform bound sp}, we obtain $(u(\cdot,t), u_t(\cdot,t)) \in \dot{H}^{s_p}\times \dot{H}^{s_p-1}$ for all time $t\in \Rm$. By completeness of the space $\dot{H}^{s_p}\times \dot{H}^{s_p-1}$ we conclude that there exists $(u_0^+, u_1^+) \in \dot{H}^{s_p}\times \dot{H}^{s_p-1}$, so that 
\[
 \limsup_{t\rightarrow +\infty} \left\|\mathbf{S}_L (-t) \begin{pmatrix} u(\cdot,t)\\ u_t(\cdot,t) \end{pmatrix} - \begin{pmatrix} u_0^+\\ u_1^+ \end{pmatrix} \right\|_{\dot{H}^{s_p} \times \dot{H}^{s_p-1} (\Rm^2)} = 0.
\]
Thus we obtain the scattering of solutions in the critical Sobolev space
\[
 \limsup_{t\rightarrow +\infty} \left\|\begin{pmatrix} u(\cdot,t)\\ u_t(\cdot,t) \end{pmatrix} - \mathbf{S}_L (t) \begin{pmatrix} u_0^+\\ u_1^+ \end{pmatrix} \right\|_{\dot{H}^{s_p} \times \dot{H}^{s_p-1} (\Rm^2)} = 0.
\]
By scattering criterion (Proposition \ref{scattering criterion}) we have $u \in L^{\frac{3}{2}(p-1)} L^{\frac{3}{2}(p-1)} (\Rm^+ \times \Rm^2)$. Next we show that the scattering also happens in the energy space. We start by applying Strichartz estimates and a fractional chain rule (Lemma \ref{chain rule}) to obtain 
\begin{align}
 \|D^{1/2} u\|_{L^6 L^6 ([t_1,t_2]\times \Rm^2)} & \leq C \left\|(u(\cdot,t_1), u_t(\cdot,t_1))\right\|_{\dot{H}^1 \times L^2} + C_1 \|D^{1/2} (-|u|^{p-1}u)\|_{L^{6/5} L^{6/5} ([t_1,t_2]\times \Rm^2)} \nonumber\\
 & \leq 2^{1/2} C E^{1/2} + C_2 \|u\|_{L^{\frac{3}{2}(p-1)} L^{\frac{3}{2}(p-1)} ([t_1,t_2]\times \Rm^2)}^{p-1} \|D^{1/2} u\|_{L^6 L^6([t_1,t_2]\times \Rm^2)} \label{Dahalf}
\end{align}
Please note that $\|D^{1/2} u\|_{L^6 L^6([t_1,t_2]\times \Rm^2)} < +\infty$ for all $t_2>t_1\geq 0$. Because
\begin{align*}
 \|D^{1/2} u\|_{L^6 L^6([t_1,t_2]\times \Rm^2)} & \leq C\left\|(u(\cdot,t_1), u_t(\cdot,t_1))\right\|_{\dot{H}^1 \times L^2} + \|-|u|^{p-1} u\|_{L^1 L^2([t_1,t_2]\times \Rm^2)}\\
 & \leq 2^{1/2} C E^{1/2} + (t_2-t_1) \|-|u|^{p-1} u\|_{L^\infty L^2([t_1,t_2]\times \Rm^2)}\\
 & \leq 2^{1/2} C E^{1/2} + (t_2-t_1) \|u\|_{L^\infty L^{2p} ([t_1,t_2]\times \Rm^2)}^p < +\infty
\end{align*}
Here we apply Strichartz estimates and use the embedding $\dot{H}^{1}\cap L^{p+1} (\Rm^2)\hookrightarrow L^{2p}(\Rm^2)$. Since $u \in L^{\frac{3}{2}(p-1)} L^{\frac{3}{2}(p-1)} (\Rm^+ \times \Rm^2)$, we may find a large time $t_1$ so that
\[
 C_2 \|u\|_{L^{\frac{3}{2}(p-1)} L^{\frac{3}{2}(p-1)} ([t_1,+\infty)\times \Rm^2)}^{p-1} < 1/2.
\]
A combination of this inequality with \eqref{Dahalf} immediately gives
\[
 \|D^{1/2} u\|_{L^6 L^6 ([t_1,t_2]\times \Rm^2)} \leq 2^{3/2} C E^{1/2},
\]
for all $t_2 > t_1$. Therefore we have
\[
 \|D^{1/2} u\|_{L^6 L^6 ([t_1,+\infty)\times \Rm^2)} < + \infty \quad \Rightarrow \quad \|D^{1/2} (-|u|^{p-1}u)\|_{L^{6/5} L^{6/5} ([t_1,+\infty)\times \Rm^2)}<+\infty.
\]
We then apply Strichartz estimates and obtain
\[ 
 \limsup_{t_1,t_2\rightarrow +\infty} \left\|\mathbf{S}_L (-t_1) \begin{pmatrix} u(\cdot,t_1)\\ u_t(\cdot,t_1) \end{pmatrix} - \mathbf{S}_L(-t_2) \begin{pmatrix} u(\cdot,t_2)\\ u_t(\cdot,t_2)\end{pmatrix}\right\|_{\dot{H}^{1} \times L^2 (\Rm^2)} = 0.
\]
Therefore the convergence 
\[
 \mathbf{S}_L (-t) \begin{pmatrix} u(\cdot,t)\\ u_t(\cdot,t) \end{pmatrix}  \rightarrow \begin{pmatrix} u_0^+\\ u_1^+ \end{pmatrix}
\]
happens not only in the space $\dot{H}^{s_p}\times \dot{H}^{s_p-1}$ but also in the space $\dot{H}^1 \times L^2$. Finally an interpolation between these two spaces implies the scattering of solutions in all spaces $\dot{H}^s \times \dot{H}^{s-1}(\Rm^2)$ with $s\in [s_p,1]$.

\section{Scattering Theory of Radial Solutions}

In this section we prove scattering theory of radial finite-energy solutions to (CP1). The general idea is the same as in higher dimensions, see 
\cite{shenenergy}. We combine a method of characteristic lines with energy distribution properties of solutions. 

\subsection{Method of Characteristic Lines}

\paragraph{Reduction to one-dimensional case} Let $u$ be a radial solution to (CP1) with a finite energy $E$. Please note that we will work as though the solutions are sufficiently smooth, otherwise we may apply standard smooth approximation techniques. We start by defining
 \begin{align*}
  &w(r,t) = r^{1/2} u(r,t);& &&\\
  &v_+(r,t) = w_t(r,t) - w_r(r,t);& &v_-(r,t) = w_t(r,t) + w_r(r,t).&
 \end{align*}
The function $w(r,t)$ satisfies a one-dimensional wave equation 
\[
 w_{tt} - w_{rr} = f(r,t) \doteq + \frac{1}{4} r^{-3/2} u - r^{1/2} |u|^{p-1} u. 
\] 
A simple calculation shows that $v_+(\cdot,t), v_-(\cdot,t) \in L_{loc}^2(\Rm^+)$ satisfy
\begin{align*}
 (\partial_t \pm \partial_r) v_\pm (r,t) = w_{tt} - w_{rr} = f(r,t).
\end{align*}
This immediately gives variation of $v_\pm$ along characteristic lines $t \pm r = \hbox{Const}$. 
\begin{align*}
  &v_+(t_2-\eta, t_2) - v_+(t_1-\eta, t_1)  = \int_{t_1}^{t_2} f(t-\eta,t) dt,& &t_2>t_1>\eta;& \\
  &v_-(s-t_2,t_2) - v_-(s-t_1,t_1)  = \int_{t_1}^{t_2} f(s-t,t) dt,& &t_1<t_2<s.&
\end{align*} 
Next we give the upper bounds of the integrals of $f$ above. According to Lemma \ref{pointwise estimate}, we have 
\[
 \int_{t_1}^{t_2} |(t-\eta)^{-3/2} u(t-\eta,t)| dt \lesssim_E \int_{t_1}^{t_2} (t-\eta)^{-\frac{3}{2}-\frac{2}{p+3}} dt \lesssim_E (t_1-\eta)^{-\frac{1}{2}-\frac{2}{p+3}}. 
\]
In addition, we may utilize energy flux formula and obtain 
\begin{align*}
 \left|\int_{t_1}^{t_2} (t-\eta)^{1/2} |u|^{p-1}u(t-\eta,t) dt\right| &\leq \left(\int_{t_1}^{t_2} (t-\eta) |u(t-\eta,t)|^{p+1} dt\right)^{\frac{p}{p+1}}\left(\int_{t_1}^{t_2} (t-\eta)^{-\frac{p-1}{2}}dt\right)^\frac{1}{p+1}\\
 & \lesssim_{p,E}  (t_1-\eta)^{-\frac{p-3}{2(p+1)}}. 
\end{align*}
Combining these two upper bounds together we have 
\[
 \left|\int_{t_1}^{t_2} f(t-\eta,t) dt \right| \lesssim_{p,E} (t_1-\eta)^{-\frac{1}{2}-\frac{2}{p+3}} + (t_1-\eta)^{-\frac{p-3}{2(p+1)}}. 
\]
Thus there exists a function $g_+(\eta) \in L_{loc}^2 (\Rm)$, so that $v_+(t-\eta,t)$ converges to $g_+(\eta)$ in $L_{loc}^2(\Rm)$ as $t\rightarrow +\infty$.  In fact we have 
\begin{equation*}
 \int_{\eta_1}^{\eta_2} \left|v_+(t-\eta,t) - g_+(\eta)\right|^2 d\eta \lesssim_{p,E} (\eta_2-\eta_1) (t-\eta_2)^{-\frac{p-3}{p+1}}, \quad t>\eta_2+1.
\end{equation*}
We apply change of variables $\eta = t-r$ and rewrite this inequality in the form 
\begin{equation} \label{variation vplus}
 \int_{r_1}^{r_2} \left|v_+(r,t) - g_+(t-r)\right|^2 dr \lesssim_{p,E} (r_2-r_1) r_1^{-\frac{p-3}{p+1}}, \quad r_1>1.
\end{equation}
A basic calculation shows that 
\[
 v_+(r,t) = r^{1/2}(u_t(r,t)-u_r(r,t))-(1/2)r^{-1/2} u(r,t).
\] 
We may apply Lemma \ref{pointwise estimate} and obtain 
\[
 \int_{r_1}^{r_2} |r^{-1/2} u(r,t)|^2 dr \lesssim_{p,E} (r_2-r_1) r_1^{-1-\frac{4}{p+3}}, \quad r_1>1.
\]
A combination of this inequality with \eqref{variation vplus} gives 
\[
  \int_{r_1}^{r_2} \left|r^{1/2}(u_t(r,t)-u_r(r,t)) - g_+(t-r)\right|^2 dr \lesssim_{p,E} (r_2-r_1) r_1^{-\frac{p-3}{p+1}}, \quad r_1>1.
\]
Next we show $g_+ \in L^2(\Rm)$. Given any $\eta_1<\eta_2$, we may utilize the inequality above with large time $t$, $r_1 = t - \eta_2$, $r_2 = t - \eta_1$ and consider the limit $t\rightarrow +\infty$:
\begin{align*}
 \int_{\eta_1}^{\eta_2} |g_+(\eta)|^2 d\eta = \lim_{t\rightarrow +\infty} \int_{t-\eta_2}^{t-\eta_1} r \left|u_t(r,t)-u_r(r,t)\right|^2 dr \lesssim_1 E.  
\end{align*}
This implies $g_+ \in L^2(\Rm)$ since $\eta_1<\eta_2$ are arbitrary constants. In summary we have 
\begin{lemma} \label{L2 estimate}
 If $u$ is a radial solution to (CP1) with a finite energy $E$, then there exist two functions $g_+, g_- \in L^2(\Rm)$ so that 
\begin{align*}
  \int_{r_1}^{r_2} \left|r^{1/2}(u_t(r,t)-u_r(r,t)) - g_+(t-r)\right|^2 dr & \lesssim_{p,E} (r_2-r_1) r_1^{-\frac{p-3}{p+1}}, \quad r_1>1;\\
  \int_{r_1}^{r_2} \left|r^{1/2}(u_t(r,t)+u_r(r,t)) - g_-(t+r)\right|^2 dr & \lesssim_{p,E} (r_2-r_1) r_1^{-\frac{p-3}{p+1}}, \quad r_1>1.
\end{align*}
\end{lemma}

\subsection{Exterior Scattering} \label{sec: exterior scattering}
Given any $\eta \in \Rm$ and large positive number $R$, we may apply Lemma \ref{L2 estimate} with $r_1 = t - \eta$, $r_2 = t+R$ and obtain 
\[
 \int_{t-\eta}^{t+R} \left|r^{1/2}(u_t(r,t)-u_r(r,t)) - g_+(t-r)\right|^2 dr \lesssim_{p,E} (R+\eta) (t-\eta)^{-\frac{p-3}{p+1}}, \quad t\gg 1.
\]
We may let $t\rightarrow +\infty$:
\begin{equation} \label{asymptotic behaviour 1}
 \lim_{t\rightarrow +\infty}\int_{t-\eta}^{t+R} \left|r^{1/2}(u_t(r,t)-u_r(r,t)) - g_+(t-r)\right|^2 dr = 0.
\end{equation}
Similarly we have 
\[
 \lim_{t\rightarrow +\infty}\int_{t-\eta}^{t+R} \left|r^{1/2}(u_t(r,t)+u_r(r,t)) - g_-(t+r)\right|^2 dr = 0.
\]
Since $g_- \in L^2(\Rm)$, we may discard $g_-(t+r)$ above and write 
\[
 \lim_{t\rightarrow +\infty}\int_{t-\eta}^{t+R} \left|r^{1/2}(u_t(r,t)+u_r(r,t)) \right|^2 dr = 0.
\]
We combine this with \eqref{asymptotic behaviour 1} and obtain 
\begin{equation}\label{asymptotic behaviour 2}
 \lim_{t\rightarrow +\infty} \int_{t-\eta}^{t+R} \left(\left|r^{1/2} u_t(r,t)-\frac{1}{2}g_+(t-r)\right|^2 + \left|r^{1/2}u_r(r,t)+\frac{1}{2}g_+(t-r)\right|^2\right) dr = 0
\end{equation}
We then observe 
\begin{itemize}
 \item Finite speed of propagation of energy implies 
 \[
  \lim_{R\rightarrow +\infty} \sup_{t\geq 0} \int_{t+R}^\infty r\left(|u_t(r,t)|^2 + |u_r(r,t)|^2\right) dr \leq  \lim_{R\rightarrow +\infty} \int_{R}^\infty r(|u_1(r)|^2 + |\partial_r u_0(r)|^2)dr = 0.
 \]
 \item Since $g_+ \in L^2(\Rm)$, we have 
 \[
  \lim_{R\rightarrow +\infty} \int_{t+R}^{\infty} |g_+(t-r)|^2 dr = \lim_{R\rightarrow +\infty} \int_{-\infty}^{-R} |g_+(\eta)|^2 d \eta = 0.
 \]
\end{itemize}
Combining these two limits with \eqref{asymptotic behaviour 2}, we obtain 
\begin{equation}\label{asymptotic behaviour 3}
 \lim_{t\rightarrow +\infty} \int_{t-\eta}^{+\infty} \left(\left|r^{1/2} u_t(r,t)-\frac{1}{2}g_+(t-r)\right|^2 + \left|r^{1/2}u_r(r,t)+\frac{1}{2}g_+(t-r)\right|^2\right) dr = 0
\end{equation}
By radiation fields, there exists a radial free wave $\tilde{u}^+$, so that 
\[
 \lim_{t\rightarrow +\infty} \int_0^{+\infty} \left(\left|r^{1/2} \tilde{u}_t^+ (r,t)-\frac{1}{2}g_+(t-r)\right|^2 + \left|r^{1/2} \tilde{u}_r^+ (r,t)+\frac{1}{2}g_+(t-r)\right|^2\right) dr = 0.
\]
Therefore we have
\[
 \lim_{t\rightarrow +\infty} \int_{t-\eta}^{+\infty} r\left(\left|u_t(r,t) - \tilde{u}_t^+ (r,t)\right|^2 + \left|u_r(r,t) - \tilde{u}_r^+ (r,t)\right|^2\right) dr = 0.
\]
This proves the exterior scattering of solutions, i.e. scattering outside any forward light cone. 

\subsection{Scattering by energy decay}

In this subsection we prove Part (b) of Theorem \ref{main 2}. Let $t>0$ be sufficiently large and $c\in (0,1)$ be a constant. We may choose $r_1=t-c \cdot t^{\frac{p-3}{p+1}}$ and $r_2 = t$ in Lemma \ref{L2 estimate} and obtain 
\begin{align*}
  \limsup_{t\rightarrow +\infty} \int_{t-c\cdot t^{\frac{p-3}{p+1}}}^t \left|r^{1/2}(u_t(r,t)-u_r(r,t)) - g_+(t-r)\right|^2 dr & \lesssim_{p,E} c;\\
  \limsup_{t\rightarrow +\infty} \int_{t-c\cdot t^{\frac{p-3}{p+1}}}^t \left|r^{1/2}(u_t(r,t)+u_r(r,t)) - g_-(t+r)\right|^2 dr & \lesssim_{p,E} c.
\end{align*}
Following the same argument as in last subsection, we have  
\[
 \limsup_{t\rightarrow +\infty} \int_{t-c\cdot t^{\frac{p-3}{p+1}}}^{t} r\left(\left|u_t(r,t) - \tilde{u}_t^+ (r,t)\right|^2 + \left|u_r(r,t) - \tilde{u}_r^+ (r,t)\right|^2\right) dr \lesssim_{p,E} c.
\]
We may combine this with the regular exterior scattering and obtain 
\begin{equation} \label{stronger exterior scattering}
 \limsup_{t\rightarrow +\infty} \int_{t-c\cdot t^{\frac{p-3}{p+1}}}^{+\infty} r\left(\left|u_t(r,t) - \tilde{u}_t^+ (r,t)\right|^2 + \left|u_r(r,t) - \tilde{u}_r^+ (r,t)\right|^2\right) dr \lesssim_{p,E} c.
\end{equation}
Next we consider the region $\{x: |x|<t-c\cdot t^{\frac{p-3}{p+1}}\}$. We utilize the conclusion of Theorem \ref{main 2}, part (c) and obtain
\[
 \lim_{t\rightarrow +\infty} \int_{|x|<t-c\cdot t^{\frac{p-3}{p+1}}} e(x,t) dx \lesssim_c \lim_{t\rightarrow +\infty} t^{\frac{4}{p+1}} \int_{|x|<t} \frac{t-|x|}{t}e(x,t) dx = 0.
\]
Please note that in the sub-conformal range our assumption $p>1+2\sqrt{3}$ guarantees that $\frac{4}{p+1} < \frac{p-3}{2}$. By Lemma \ref{hollow center}, we also have 
\[
 \lim_{t\rightarrow +\infty} \int_{|x|<t-c\cdot t^{\frac{p-3}{p+1}}} \left(|\nabla \tilde{u}^+(x,t)|^2 + |\tilde{u}_t^+(x,t)|^2\right) dx = 0.
\]
Combining these two limits we obtain 
\begin{align*}
  \lim_{t\rightarrow + \infty} \int_{|x|<t-c\cdot t^{\frac{p-3}{p+1}}} \left(|\nabla \tilde{u}^\pm (x,t)-\nabla u(x,t)|^2+|\tilde{u}_t^\pm (x,t)-u_t(x,t)|^2\right) dx = 0.
\end{align*}
We may combine this with stronger exterior scattering \eqref{stronger exterior scattering} to conclude
\[
 \limsup_{t\rightarrow + \infty} \int_{\Rm^2} \left(|\nabla \tilde{u}^\pm (x,t)-\nabla u(x,t)|^2+|\tilde{u}_t^\pm (x,t)-u_t(x,t)|^2\right) dx. \lesssim_{p,E} c.
\]
Finally we make $c\rightarrow 0^+$ and finish the proof. 

\section{Appendix} 
In this appendix we give a brief proof of a Morawetz estimate for solutions to 2D wave equation. This kind of Morawetz estimates were first introduced by Nakanishi. For convenience we use the same notation as in Nakanishi \cite{morawetzNLKG}:
\begin{align*}
 &\lambda = \sqrt{t^2+r^2};& &\Theta = \frac{(-t,x)}{\lambda};& &g = \frac{d-1}{2\lambda} + \frac{t^2-r^2}{2\lambda^3};& \\
 &m_h = \Theta \cdot (u_t, \nabla u) + ug;& &l(u) = \frac{|\nabla u|^2}{2} - \frac{|u_t|^2}{2} + \frac{|u|^{p+1}}{p+1};& &\Box = \partial_t^2 - \Delta;&
\end{align*}
and $(\partial_0, \partial_1, \cdots, \partial_d) = (-\partial^0, \partial^1, \cdots, \partial^d) = (\partial_t, \nabla)$. Then we have an identity
\begin{align*}
 (\Box u + |u|^{p-1} u)m_h = &\sum_{\alpha=0}^d \partial_\alpha \left(-m_h \partial^\alpha u + l(u) \Theta_\alpha + \frac{|u|^2}{2} \partial^\alpha g\right) \\
 & \quad + \frac{|\slashed{\nabla} u|^2}{\lambda} + \frac{|r u_t+t u_r|^2}{\lambda^3} + \frac{p-1}{p+1} |u|^{p+1} g + \frac{|u|^2}{2} \Box g.
\end{align*}
A basic calculation shows that
\[
 \Box g = \frac{(d-3)(d+3)}{2\lambda^3} + 3(d-1)\frac{t^2-r^2}{\lambda^5}+15\frac{(t^2-r^2)^2}{2\lambda^7}.
\]
Thus we may integrate in the region $\Rm^2 \times [1,T]$ and obtain
\begin{align}
 &\left. \int_{\Rm^2} \left(-m_h u_t+\frac{t}{\lambda} l(u) + \frac{|u|^2}{2} g_t\right) dx \right|_{t=1}^{t=T} \nonumber \\
 & \qquad = \int_{1}^{T} \int_{\Rm^2} \left(\frac{|\slashed{\nabla} u|^2}{\lambda} + \frac{|r u_t+t u_r|^2}{\lambda^3} + \frac{p-1}{p+1} \cdot \frac{t^2 |u|^{p+1}}{\lambda^3} + \frac{|u|^2}{2} \Box g\right) dx dt. \label{second Morawetz}
\end{align}
In order to deal with the terms involving $|u|^2$, we need to apply H\"{o}lder's inequality
\begin{align*}
 \int_{\Rm^2} \frac{|u(x,t)|^2}{\lambda^2} dx \leq \left(\int_{\Rm^2} |u|^{p+1} dx\right)^{\frac{2}{p+1}} \left(\int_{\Rm^2} \lambda^{-\frac{2(p+1)}{p-1}} dx\right)^{\frac{p-1}{p+1}} \lesssim_p t^{-\frac{4}{p+1}} E^{\frac{2}{p+1}}.
\end{align*}
Next we observe the facts $|g_t| \lesssim \lambda^{-2}$, $|\Box g| \lesssim \lambda^{-3}$ and obtain
\begin{align*}
 \left|\int_{\Rm^2} \left(-m_h u_t+\frac{t}{\lambda} l(u) + \frac{|u|^2}{2} g_t\right) dx\right| & \lesssim_p \int_{\Rm^2} \left(|\nabla u|^2 + |u_t|^2 + |u|^{p+1} + \frac{|u|^2}{\lambda^2} \right) dx\\
 & \lesssim_p E + E^{\frac{2}{p+1}},
\end{align*}
and 
\[
 \int_{1}^T \int_{\Rm^2} \frac{|u|^2}{2} |\Box g| dx dt \leq \int_1^T \int_{\Rm^2} \frac{|u|^2}{t\lambda^2} dx dt \lesssim_p E^{\frac{2}{p+1}} \int_1^T t^{-1-\frac{4}{p+1}} dt \lesssim_p E^{\frac{2}{p+1}}.
\]
We plug these upper bound in \eqref{second Morawetz} and conclude
\[
 \int_{1}^{T} \int_{\Rm^2} \left(\frac{|\slashed{\nabla} u|^2}{\lambda} + \frac{|r u_t+t u_r|^2}{\lambda^3} + \frac{t^2 |u|^{p+1}}{\lambda^3}\right)  dx dt \lesssim_p E + E^{\frac{2}{p+1}}.
\]

\section*{Acknowledgement}
The second author is financially supported by National Natural Science Foundation of China Projects 12071339, 11771325.

\end{document}